\documentclass[12pt]{amsart}
\usepackage{amsmath}
\usepackage{amsfonts, amssymb, euscript, mathrsfs, stackrel}
\usepackage{amsthm, upref, courier}
\usepackage{graphicx}
\usepackage[usenames, dvipsnames]{color}

\usepackage{tikz}
\usepackage{enumerate}
\usepackage[margin=0.9in]{geometry}
\usepackage{aliascnt}
\usepackage{hyperref}
\usepackage{verbatim}

\allowdisplaybreaks[4]

\newcommand{\cl}{\ell}

\newcommand{\CO}{\mathcal O}
\newcommand{\BR}{\mathbb{R}}

\newcommand{\al}{\alpha}
\newcommand{\be}{\beta}
\newcommand{\ga}{\gamma}
\newcommand{\de}{\delta}

\newcommand{\ME}{\mathbf E}

\newcommand{\fn}{\mathbf{n}}
\newcommand{\CF}{\mathcal F}

\newcommand{\CG}{\mathcal G}
\newcommand{\MP}{\mathbf P}

\newcommand{\CA}{\mathcal A}

\newcommand{\CK}{\mathcal K}

\newcommand{\Oa}{\Omega}
\newcommand{\oa}{\omega}
\newcommand{\si}{\sigma}

\newcommand{\pa}{\partial}
\renewcommand{\phi}{\varphi}

\newcommand{\eps}{\varepsilon}
\newcommand{\la}{\lambda}
\newcommand{\Ra}{\Rightarrow}

\newcommand{\ol}{\overline}

\newcommand{\norm}[1]{\left\lVert#1\right\rVert}
\renewcommand{\comment}[1]{}

\newcommand{\cm}{\mathfrak{m}}
\newcommand{\mg}{\mathfrak{g}}

\newcommand{\md}{\mathrm{d}}

\DeclareMathOperator{\tr}{tr}

\DeclareMathOperator{\mes}{mes}
\DeclareMathOperator{\Int}{int}
\DeclareMathOperator{\dist}{dist}
\DeclareMathOperator{\CU}{\mathcal U}

\begin{document}

\theoremstyle{plain}
\newtheorem{thm}{Theorem}[section]
\newtheorem*{thmnonumber}{Theorem}
\newtheorem{lemma}[thm]{Lemma}
\newtheorem{prop}[thm]{Proposition}
\newtheorem{cor}[thm]{Corollary}
\newtheorem{open}[thm]{Open Problem}

\theoremstyle{definition}
\newtheorem{defn}{Definition}
\newtheorem{asmp}{Assumption}
\newtheorem{notn}{Notation}
\newtheorem{prb}{Problem}

\theoremstyle{remark}
\newtheorem{rmk}{Remark}
\newtheorem{exm}{Example}
\newtheorem{clm}{Claim}

\author{Andrey Sarantsev}

\title[Penalty Method for Obliquely Reflected Diffusions]{Penalty Method for Obliquely Reflected Diffusions}

\address{University of Nevada, 1664 North Virginia Street, Reno, NV 89557}

\email{asarantsev@unr.edu}



\date{Received March 24, 2020. Accepted July 19, 2021.}

\begin{abstract}
Take a multidimensional normally or obliquely reflected diffusion in a smooth domain. Approximate it by solutions of stochastic differential equations without reflection using the penalty method. That is, we approximate the reflection term with an additional drift term. In the existing literature, usually a specific approximating sequence is provided in order to prove existence of a reflected diffusion. In this article, we provide general sufficient conditions on the approximating coefficients. \end{abstract}

\maketitle

\section{Introduction} 

\thispagestyle{empty}

If $W = (W(t),\, t \ge 0)$ is a Brownian motion on the real line, then $|W| = (|W(t)|,\, t \ge 0)$ is a reflected Brownian motion on the positive half-line $\mathbb R_+ := [0, \infty)$. This process $|W|$ behaves as a Brownian motion as long as it stays away from the origin on the positive half-line. As it hits the origin, it is reflected back to the positive half-line. We can generalize this to a reflected Brownian motion $Z = (Z(t),\, t \ge 0)$ on the half-line with a drift coefficient $b$ and a diffusion coefficient $\si^2$. 

\smallskip

More generally, we can make $b$ and $\sigma$ to be dependent on the current position $Z(t)$ of the process $Z$. In this case, this process behaves as a solution to a one-dimensional stochastic differential equation with drift coefficient $b(\cdot)$ and diffusion coefficient $\sigma^2(\cdot)$, as long as it stays away from the origin. This process $Z$ is then called a {\it reflected diffusion on the half-line} with {\it drift coefficient} $b(\cdot)$ and {\it diffusion coefficient} $\sigma^2(\cdot)$.  The concept of a reflected diffusion on the half-line was introduced in the  papers \cite{McKean1963, Skorohod1961a, Skorohod1961b}.

\smallskip

Yet more generally, we consider a reflected Brownian motion and a reflected diffusion in a domain (an open connected subset) in a Euclidean domain. Let us informally describe this process, with formal definition postponed until Section 2. Throughout this paper, fix a dimension $d \ge 2$. Take a domain (an open connected subset) $D \subseteq \BR^d$ with $C^2$ (twice continuously differentiable) boundary $\pa D$, and denote its closure by $\ol{D}$. For every point $x \in \pa D$ on the boundary, let $\fn(x)$ be the inward unit normal vector. 

\smallskip

Take a continuous {\it reflection field} $r : \pa D \to \BR^d$ with the property $r(x)\cdot\fn(x) = 1$ for $x \in \pa D$. If $r(x) \equiv \fn(x)$, we say this corresponds to a {\it normal reflection}. Otherwise, this is called an {\it oblique reflection}. Take a {\it drift vector field} $b : \ol{D} \to \BR^d$. For every $x \in \ol{D}$, let $A(x)$ be a symmetric positive definite $d\times d$-matrix. 
Let us describe the concept of a {\it multidimensional obliquely reflected diffusion} $Z = (Z(t), t \ge 0)$ in $\ol{D}$, with {\it drift vector field} $b$, {\it covariance matrix field} $A$, and {\it reflection field} $r$. This is a process which:

\smallskip

(a) behaves as a $d$-dimensional solution of a stochastic differential equation with certain drift vector field $b(\cdot)$ and certain covariance matrix field $A(\cdot)$, as long as it is inside $D$;

\smallskip

(b) as this process hits the boundary $\pa D$ at a point $x \in \pa D$, it is reflected back inside the domain $D$ according to the vector $r(x)$. %The field $r(\cdot)$ is called a {\it reflection vector field}. 

\smallskip

If the fields $b$ and $A$ are constant, this process is called a {\it reflected Brownian motion in} $\ol{D}$ with {\it drift vector} $b$ and {\it covariance matrix} $A$. 

\smallskip

Reflected Brownian motions and reflected diffusions were extensively studied over the past half-century. The construction and study of reflected diffusions in multidimensional domains were done in numerous articles, including \cite{BassHsu, Costas, DW1995, DupuisIshii1991, DI, Fukushima, HR1981a,  LS1984, PilipenkoBook, Saisho, SV1971, Tanaka1979, Watanabe1971a, Watanabe1971b}, including the case of oblique reflection. Note that it is possible to define a reflected diffusion in a non-smooth domain, for example a convex polyhedron. However, we shall limit ourselves to a smooth domain, for reasons explained below. 

\smallskip

Applications of reflected Brownian motions and reflected diffusions include  queueing theory \cite{HW1987b}, stochastic finance \cite{BFK, OrderBook}, and transport processes in chemistry \cite{Transport}. 

\smallskip

The goal of this paper is to approximate a reflected diffusion $Z = (Z(t), t \ge 0)$ by a non-reflected diffusion $X = (X(t), t \ge 0)$. We do this by replacing the reflection by an additional drift term $f(X(t))\,\md t$. This gives us $X$ which is a solution to a stochastic differential equation with drift vector field $b + f$ and covariance matrix $A$. This additional drift vector field $f(\cdot)$ is custom-designed so that it penalizes the process $X$ for wandering away from the domain $D$. This motivates the name {\it penalty method} for this type of approximation. We state our main results in Theorems~\ref{thm:1} and~\ref{thm:2} in general terms, providing conditions on the vector field $f(\cdot)$ which are sufficient for the penalty method to work. We also provide a class of concrete examples in Theorem~\ref{thm:3}. Let us informally explain the statement of Theorems~\ref{thm:1} and~\ref{thm:2}. We present a weak approximation of the process $Z$ by a sequence $(X_n)_{n \ge 1}$ of solutions of stochastic differential equations with drift vector fields $b + f_n$ and covariance matrix $A$. Here, $f_n(X_n(t))\,\md t$ gives us the penalty term emulating the reflection. To this end, we need to impose the following conditions on the sequence $(f_n)_{n \ge 1}$ of drift vector fields:

\smallskip

(a) Locally uniformly on $D$, $f_n(x) \to 0$ as $n \to \infty$, to avoid an additional drift inside $D$;

\smallskip

(b) For large $n$, and for points $x$ close to the boundary $\pa D$, the drift vector $f_n(x)$ has the direction which is close to the direction of the reflection field $r(y)$, for $y \in \pa D$ close to $x$:
$$
\frac{f_n(x)}{\norm{f_n(x)}} \approx \frac{r(y)}{\norm{r(y)}}\ \ \mbox{for}\ \ x \approx y \in \pa D,
$$
where $x \approx y$ stands for $x$ being close to $y$ in the Euclidean distance;

\smallskip

(c) For large $n$ and for $x$ close to the boundary $\pa D$, the magnitude $\norm{f_n(x)}$ is large enough to prevent the process $X_n$ from wandering far away from the domain $D$. 

\smallskip

Condition (c) ensures that the process $X_n$ is penalized for straying far away from $D$. As mentioned before, this is the reason we call this the {\it penalty method}. 

\begin{rmk} Actually, in Theorems~\ref{thm:1}, ~\ref{thm:2}, we allow for an additional degree of freedom: instead of the covariance matrix field $A$, we can have covariance matrix field $A_n$, with $A_n \to A$ locally uniformly in an open neighborhood of $\ol{D}$. 
\end{rmk}

The idea of the penalty method is not new. For example, penalty method was used in the classic paper \cite{LS1984} to prove existence of obliquely reflected diffusions. We shall survey some other articles on the penalty method below. 

\smallskip

However, all these articles provide a concrete example of an approximating sequence. We tried to answer the following question: What are the most general conditions under which the sequence of penalized non-reflected diffusions weakly converges to the given reflected diffusion? To the best of our knowledge, no article exists which provides a general answer to such question. This is the goal of the current article. 

\smallskip

One possible future application of such general result might include efficient simulation of reflected diffusions. Wong-Zakai approximation (Euler-type schemes) for normally reflected diffusions were studied in \cite{Aida, Petterson3, Slo7}, and for obliquely reflected diffusions in \cite{Evans}. One can try to simulate reflected diffusions by choosing an appropriate penalty drift, and then simulating the resulting stochastic differential equation without reflection via the standard Wong-Zakai procedure. The next question is what penalty term from our results is better for such simulation. We leave this topic for future research. 

\smallskip

\subsection{Review of the existing literature on the penalty method} Let us now survey existing literature on the penalty method, which applies both to reflected stochastic differential equations, and to their deterministic analogue, the {\it Skorohod problem}. The literature on normal reflection is more extensive. Probably the most popular penalty function in the literature is 
\begin{equation}
\label{eq:popular}
f_n(x) := n(x - \Pi(x)).
\end{equation}
Here, $\Pi(x)$ is the {\it projection} of $x$ onto $D$: the closest point on $\ol{D}$ to $x$. This function is used in the papers \cite{Aux1, Slo4, Menaldi, Petterson1, Petterson2, Pilipenko-Penalty, Slo1, Storm} for normally reflected diffusion in a convex domain, and in the papers \cite{Slo2, Slo3, Men1, Men2}  for stochastic differential equations with jumps. The foundational article for multiple dimension case is \cite{SV1971}, where authors approximate the process the sequence of continuous time Markov chains. The diffusion is allowed to be time-inhomogeneous. The paper \cite{Cepa} applies the function~\eqref{eq:popular} to a multivalued Skorohod problem.  The penalty method with penalty function~\eqref{eq:popular} was also used in \cite{Karoui} to approximate reflected backward stochastic differential equations (BSDE) on the real line with non-reflected BSDE and to prove existence of a solution of a reflected BSDE. In the multidimensional setting, similar work was carried out in \cite{Slo5, Slo6}. The papers \cite{Portenko1979, Portenko1979a} deal with similar situations, considering a Brownian motion with a drift coefficient  which is, in fact, a distribution. 

\smallskip

For oblique reflection, we need to use a different penalty drift than~\eqref{eq:popular}, or to modify it. The paper \cite{Shalaumov} contains the case of the halfspace $D := \BR^{d-1}\times (0, \infty)$ with general oblique reflection, with the penalty function $f_n(x) = nB(x)$ for a suitably chosen vector field $B :\BR^d \to \BR^d$, which is continuous and is equal to $0$ in $D$. The papers \cite{PardouxWilliams, WilliamsZheng1990} apply penalty method for obliquely reflected diffusions in smooth domains. However, they consider a class of {\it conormal} reflection fields, which includes the normal reflection, but does not cover all directions of oblique reflection. Also, they do this for the case when the reflected diffusion is in its stationary distribution. The paper \cite{SoftlyRBM} also deals with stationary distributions for the penalized Brownian motion in a convex polyhedron, which is intended to approximate a semimartingale reflected Brownian motion in this polyhedron; the authors call this {\it soft reflection}. The corresponding penalty function in this case is $U(x) = e^{-\la x}$ for $\la > 0$, multiplied by a proper vector to make for oblique reflection. Ranked Brownian particles on the real line reflecting upon each other are studied in \cite{Spohn}; they form an obliquely reflected Brownian motion in a Weyl chamber $\{y_1 \le \ldots \le y_d\}$, and are approximated using the penalty method in \cite{Spohn}. %, see \cite{Spohn} with general penalty function $V$, multiplied by an appro

\smallskip

This paper is somewhat close in its spirit to \cite{Kang}, which shows invariance principle for an obliquely reflected Brownian motion in a piecewise smooth domain. However, our setting is different: an obliquely reflected diffusion instead of a Brownian motion, and a smooth instead of a piecewise smooth domain. In addition, the conditions on the approximating sequence $(X_n)_{n \ge 1}$ are quite different in \cite{Kang} from here. In effect, in the paper \cite{Kang}, they presume that the condition (c) mentioned above already holds. This makes the paper \cite{Kang} closer to a different paper \cite{MyOwn9} of ours, which studies approximation of reflected diffusions by other reflected diffusions. 

\smallskip

Let us stress that, unlike in some of the papers cited above, we are not interested in using penalty method to prove existence of a reflected diffusion. Rather, we already assume weak existence and uniqueness in law of this obliquely reflected diffusion. 

\smallskip

Our companion paper \cite{MyOwn7} deals with the case of reflected diffusions on the positive half-line $\BR_+ := [0, \infty)$, where we are able to prove stronger results than in the current paper. In that case, there is no concept of oblique reflection; this reflection  can only be normal. 

\smallskip

We can prove the results of the current paper for a wider class of domains: {\it piecewise smooth domains}, under a restrictive condition that the reflected diffusion a.s. does not hit non-smooth parts of the boundary. There are some sufficient conditions for this, see \cite{MyOwn3, Williams1987}. However, this question is not well explored, and for this reason we state results only for smooth domains.

\subsection{Organization of the paper} Section 2 contains notation and definition of a reflected diffusion, as well as an existenceand uniqueness result from \cite{Saisho}. Section 3 contains the main results: Theorems~\ref{thm:1},~\ref{thm:2},~\ref{thm:3}. Sections 4-6 are devoted to the proofs of each of these three results. The Appendix contains some technical lemmata.

\section{Background} 

\subsection{Notation} For a vector or a matrix $A$, we denote its transpose by $A'$. We think of vectors from $\BR^d$ as column-vectors. The symbol $\BR^{d\times d}$ denotes the set of $d\times d$-matrices with real-valued entries. For two vectors $a = (a_1, \ldots, a_d)'$ and $b = (b_1, \ldots, b_d)'$ from $\BR^d$, we define their {\it dot product} by $a\cdot b = a_1b_1 + \ldots + a_db_d$, and the {\it Euclidean norm} of $a$ by $\norm{a} = \left[a_1^2 + \ldots + a_d^2\right]^{1/2}$. For $x \in \BR^d$ and $r > 0$, we denote the closed ball $B(x, r) = \{y \in \BR^d\mid \norm{x-y} \le r\}$. We denote by $\de_{ij}$ the Kronecker delta symbol, and by $e_i = (\de_{i1}, \ldots, \de_{id})'$ the $i$th standard unit vector in $\BR^d$, for $i = 1, \ldots, d$. For two subsets $E, F \subseteq \BR^d$, let 
$$
\dist(a, E) = \inf\limits_{x \in E}\norm{a - x}\ \ \mbox{and}\ \ \dist(E, F) = \inf\limits_{x \in E, y \in F}\norm{x - y}
$$
be the distance from $a$ to $E$ and the distance between $E$ and $F$, respectively. The interior and the closure of $E$ in the topology generated by Euclidean norm are denoted by $\Int E$ and $\ol{E}$, respectively. The complement $\BR^d\setminus E$ of $E$ is denoted by $E^c$.  The symbol $\mes(E)$ denotes the $d$-dimensional Lebesgue measure of the set $E$. The $d\times d$  identity matrix is denoted by $I_d$. We let $\BR_+ := [0, \infty)$. We denote weak convergence of probability measures and random variables by the symbol $\Ra$, and equality in law by the symbol $\stackrel{d}{=}$. The symbol $C([0, T], \BR^d)$ stands for the space of continuous functions $[0, T] \to \BR^d$. For $d = 1$, we simply write $C([0, T], \BR^d) = C[0, T]$. For a subset $E \subseteq \BR^d$, the symbol $C^r(E)$ stands for $r$ times continuously differentiable functions $E \to \BR$; this includes $r = \infty$. For the rest of the article, we define all random objects on a filtered probability space 
$$
(\Oa, \CF, (\CF_t)_{t \ge 0}, \MP),
$$
with the filtration $(\CF_t)_{t \ge 0}$ satisfying the usual conditions. We say that an $\BR^d$-valued process $X = (X(t), t \ge 0)$ is {\it continuous} if $\MP$-a.s. the trajectory $t \mapsto X(t)$ is continuous on $\BR_+$. For a function $f : E \to \BR$, where $E \subseteq \BR^d$, denote by $\nabla f$ the gradient of $f$, and by $D^2f$ the second derivative matrix of $f$. For an $\BR^d$-valued stochastic process $X = (X(t), t \ge 0)$ and a subset $E \subseteq \BR^d$, define the {\it exit time} of $X$ from $E$:
$$
\tau^X_E := \inf\{t \ge 0\mid X(t) \notin E\},\ \ \mbox{with}\ \ \inf\varnothing = \infty.
$$
We define the {\it process stopped at exiting} $E$ as follows:
\begin{equation}
\label{eq:stopped}
X^E = (X^E(t), t \ge 0),\ \ X^E(t) = X(t\wedge\tau^X_E).
\end{equation}
Define the {\it modulus of continuity} for a function $f : \BR_+ \to \BR^d$: for $T> 0$ and $\de > 0$, 
$$
\oa(f, [0, T], \de) := \sup\limits_{\substack{t, s \in [0, T]\\ |t - s| \le \de}}\norm{f(t) - f(s)}.
$$

\subsection{Definition of a reflected diffusion} Now we shall formally define the concepts of an obliquely reflected Brownian motion and an obliquely reflected diffusion, which were informally described in the Introduction. We use the notation from the Introduction, including $D, b, A, r$. For every $x \in \ol{D}$, take a nonsingular $d\times d$-matrix $\si(x)$ such that $A(x) \equiv \si(x)\si'(x)$. Take a $d$-dimensional Brownian motion $W = (W(t), t \ge 0)$. Fix a point $z_0 \in \ol{D}$. We impose the following assumption on the reflection field $r : \pa D \to \BR^d$.

\begin{asmp} The function $r : \pa D \to \BR^d$ is continuous, and $r(x)\cdot\fn(x) \equiv 1$ on $\pa D$. 
\label{asmp:r} 
\end{asmp}

\begin{defn} Consider a $\ol{D}$-valued continuous adapted process $Z = (Z(t), t \ge 0)$ and a real-valued continuous nondecreasing adapted process $\cl = (\cl(t), t \ge 0)$, with $\cl(0) = 0$, which can increase only when $Z(t) \in \pa D$, such that
\label{def:SDE-R-ps}
for $t \ge 0$, we have:
\begin{equation}
\label{eq:def-SRBM}
Z(t) = z_0 + \int_0^{t}b(Z(s))\,\md s + \int_0^t\si(Z(s))\,\md W(s) + \int_0^tr(Z(s))\,\md\cl(s).
\end{equation}
Then $Z$ is called a {\it reflected diffusion} in $\ol{D}$ with {\it drift vector field} $b$, {\it covariance matrix field} $A$, and {\it reflection field} $r$ on the boundary $\pa D$, {\it starting from} $z_0$.  
\end{defn}

Rewrite the stochastic equation~\eqref{eq:def-SRBM} governing the process $Z$, in the differential form:
\begin{equation}
\label{eq:most-general}
\md Z(t) = b(Z(t))\,\md t + \si(Z(t))\,\md W(t) + r(Z(t))\,\md\cl(t).
\end{equation}

%Sometimes it is called a {\it semimartingale reflected Brownian motion} (SRBM), to emphasize the semimartingale form of~\eqref{eq:def-SRBM}. 

As explained in the Introduction, usually in the literature they prove existence of such a process by presenting a sequence of non-reflected diffusions which converge to this reflected diffusions. This was the original use of a penalty method. In our paper, we are trying to find the most general conditions for applicability of the penalty method. Thus we assume that the reflected diffusion exists and is unique, at least in the weak sense. 

\begin{asmp} The process $Z$ in Definition~\ref{def:SDE-R-ps} exists in the weak sense: There exists a filtered probability space $(\Oa, \CF, (\CF_t)_{t \ge 0}, \mathbb P)$ with a $d$-dimensional $(\CF_t)_{t \ge 0}$-Brownian motion $W$, and with a process $Z$ satisfying Definition~\ref{def:SDE-R-ps}. In addition, this process is unique in law.
\label{asmp:existence}
\end{asmp}

Sufficient conditions for Assumption~\ref{asmp:existence} to hold can be found in papers \cite{DI, Fradon, LS1984, Saisho}. As an example, let us provide a sufficient condition (which is by no means necessary) for Assumption~\ref{asmp:existence} to hold, adapted from \cite[Theorem 6.11]{MenaldiBook}: $D$ is a $C^3$ simply connected domain with connected and oriented boundary $\pa D$, $r$ is $C^2$, $b$ and $\si$ are Lipschitz continuous. 

\section{Main Results} Back in the setting of Definition~\ref{def:SDE-R-ps}, let us impose the following assumptions.

\begin{asmp} There exists an open neighborhood $\CU$ of $\ol{D}$ such that the drift vector field $b$ and the matrix-valued function $\si$ can be extended to continuous functions on $\CU$. The initial condition $z_0$ lies away from the boundary $\pa D$: $z_0 \in D$.
\label{asmp:continuous}
\end{asmp}

Define the {\it signed distance} $\phi : \BR^d \to \BR$ to the boundary $\pa D$:
\begin{equation}
\label{eq:phi}
\phi(x) = \begin{cases}
\dist(x, \pa D),\ x \in D;\\
0,\ x \in \pa D;\\
-\dist(x, \pa D),\ x \notin \ol{D}.
\end{cases}
\end{equation}

The next technical lemma follows from the results of \cite{Foote, Krantz}. 

\begin{lemma}
\label{lemma:aux-ps} 
There exists an open neighborhood $\CU_0$ of $\pa D$ such that:

\smallskip

(a) the signed distance function $\phi$ is $C^2$ on $\CU_0$;

\smallskip

(b) for every $x \in \CU_0$, there exists a unique $y(x) \in \pa D$ such that $\dist(x, \pa D) = \norm{x - y(x)}$;

\smallskip

(c) the function $x \mapsto y(x)$ is continuous on $\CU_0$. 
\end{lemma}

Without loss of generality, we can assume $\CU_0 \subseteq \CU$. 

\smallskip

Define a sequence of SDEs whose solutions weakly approximate this reflected diffusion $Z$. For each $n = 1, 2, \ldots$ take a drift vector field $f_n : \CU \to \BR^d$ and a matrix field $\si_n : \CU \to \BR^{d\times d}$ such that for all $n$ and $x$, the matrix $\si_n(x)$ is nonsingular. Next, take an initial condition $z_n \in \CU$. Consider a standard $d$-dimensional Brownian motion $W_n = (W_n(t), t \ge 0)$. Define an $\BR^d$-valued continuous adapted process $X_n = (X_n(t), t \ge 0)$ such that (a) $X_n(0) = z_n$; (b) $X_n$ is stopped when it exits $\CU$; (c) as long as this process is in $\CU$, it satisfies the SDE:
$$
\md X_n(t) = \left[f_n(X_n(t)) + b(X_n(t))\right]\md t + \si_n(X_n(t))\md W_n(t).
$$

\begin{asmp} For each $n$, the process $X_n$ exists in the weak sense and is unique in law. 
\label{asmp:existence-n}
\end{asmp} 

\begin{rmk} 
If we can take $\CU = \BR^d$, then we do not need to stop the diffusion $X_n$. We provide the statement in this form for the sake of generality. 
\end{rmk}

Now, we state conditions on the sequence $(X_n)_{n \ge 1}$ of diffusions to weakly approximate the reflected diffusion $Z$. The first condition, stated as a definition, stipulates that as $n \to \infty$, the processes $X_n$ are less and less likely to stray far away from $D$. Recall the definition of the stopped process $X^E$ from~\eqref{eq:stopped} for a subset $E \subseteq \BR^d$. 

\begin{defn} We say that the sequence $(X_n)_{n \ge 1}$ is {\it asymptotically in} $\ol{D}$ if for every $T > 0$ and $\eta > 0$, and for every compact subset $\CK \subseteq \CU$, we have: 
\begin{equation}
\label{eq:asymp-X-n}
\lim\limits_{n \to \infty}\MP\Bigl(\min\limits_{0 \le t \le T}\phi\left(X_n^{\CK}(t)\right) > -\eta\Bigr) = 1.
\end{equation}
\label{def:asymp}
\end{defn}

The next definition states that for a large $n$, the drift vector field $f_n$ emulates correct direction of reflection at the boundary. More precisely, we need 
$$
\frac{f_n(x)}{\norm{f_n(x)}} \approx \frac{r(y)}{\norm{r(y)}}\ \ \mbox{for}\ \ x \approx y \in \pa D\ \ \mbox{and large}\ \ n.
$$
In fact, we can weaken this condition by considering only the values of $x$ such that $\norm{f_n(x)} \ge \eps$, where $\eps > 0$ is any fixed number. For every subset $E \subseteq \BR^d$ and a $\de > 0$, define 
\begin{equation}
\label{eq:K-de}
E(\de) := \{x \in E\mid \dist(x, \pa D) < \de\} = \{x \in E\mid -\de < \phi(x) < \de\};
\end{equation}
\begin{equation}
\label{eq:tilde-E}
\tilde{E}(\de) := \{x \in E\mid \phi(x) > -\de\}.
\end{equation}

%Extend the reflection field $r : \pa D \to \BR^d$ to any continuous function $r : \pa D \to \BR^d$ such that $r(x)\cdot\fn(x) = 1$ for all $x \in \pa D$. This is possible by Tietze's extension theorem, see for example \cite{Folland}. 

\begin{defn} The sequence $(f_n)_{n \ge 1}$ of drift vector fields is said to {\it emulate the reflection field} $r$ if for every compact subset $\CK \subseteq \CU_0$ and every $\eps > 0$, we have:
$$
\lim\limits_{\de \to 0}\varlimsup\limits_{n \to \infty}\sup\limits_{\substack{x \in \CK(\de)\\ \norm{f_n(x)} \ge \eps}}\norm{\frac{f_n(x)}{\norm{f_n(x)}} - \frac{r(y(x))}{\norm{r(y(x))}}} = 0.
$$
\label{def:direction}
\end{defn}

\begin{rmk} By Lemma~\ref{lemma:aux} below, for small enough $\de$, $y(x)$ is well-defined for $x \in \CK(\de)$. 
\end{rmk}

Let us state the first of the three main results of this paper.  

\begin{thm}
\label{thm:1}
Under Assumptions 1, 2, 3, 4, suppose that:

\smallskip

(a) the sequence $(X_n)_{n \ge 1}$ of stochastic processes is asymptotically in $\ol{D}$;

\smallskip

(b) the sequence $(f_n)_{n \ge 1}$ of drift vector fields emulates the reflection field $r$;

\smallskip

(c) as $n \to \infty$, locally uniformly in $D$, we have: $f_n \to 0$;

\smallskip

(d) as $n \to \infty$, locally uniformly in $\CU$, we have: $\si_n \to \si$;

\smallskip

(e) as $n \to \infty$, we have: $z_n \to z_0$.

\smallskip

Then, as $n \to \infty$, we have: $X_n \Ra Z$ in $C([0, T], \BR^d)$ for every $T > 0$.
\end{thm}

Let us provide a sufficient condition for the sequence $(X_n)_{n \ge 1}$ to be asymptotically in $\ol{D}$. To this end, the magnitude $\norm{f_n(x)}$ should be large enough to keep the process $X_n$ inside $\ol{D}$. For every subset $E \subseteq \CU$, every $n = 1, 2, \ldots$ and $s \in \BR$, define (with 
$\inf\varnothing := \infty$). 

\begin{equation}
\label{eq:m-k-n}
\cm_{E, n}(s) := \inf\limits_{\substack{x \in E\\ \phi(x) = s}}\norm{f_n(x)}.
\end{equation}

%Similarly to~\eqref{eq:m-k-n}, for every $i = 1, \ldots, m$, $n = 1, 2, \ldots$, define the function
%$$
%\cm_{\CK, i, n}(s) := \inf\limits_{\substack{x \in \CK\\ \phi_i(x) = s}}\norm{f_n(x)}.
%$$

\begin{defn} A sequence $(g_n)_{n \ge 1}$ of functions $g_n : \BR \to \BR_+$ {\it has a spike at zero} if there exists an $\eps_0 > 0$ such that for every $\eps \in (0, \eps_0]$, we have: 
$$
\lim\limits_{n \to \infty}\int_{-\eps}^{\eps}g_n(x)\,\md x = \infty.
$$
%We say that this sequence {\it slides away from positive support} if there exists a sequence $(\eps_n)_{n \ge n_0}$ of positive numbers such that $\eps_n \to 0$ as $n \to \infty$, and $g_n(x) = 0$ for $x > \eps_n$. 
If, in addition, $g_n \to 0$ uniformly on any $[a, b] \subseteq (0, \infty)$, then the sequence $(g_n)_{n \ge 1}$ is {\it singular}. 
\end{defn}

\begin{exm} Take a function $h: \BR \to \BR_+$ with $h(x) = 0$ for $x \ge c > 0$, and with $\int_{\BR}h(x)\md x > 0$. Consider two sequences $(a_n)_{n \ge 1}$ and $(c_n)_{n \ge 1}$ of positive real numbers, which satisfy:
$$
\lim\limits_{n \to \infty}a_n = \infty,\ \ \lim\limits_{n \to \infty}c_n = \infty,\ \ \lim\limits_{n \to \infty}a_n/c_n = \infty.
$$
It is an easy exercise to check that the following sequence of functions is singular:
\begin{equation}
\label{eq:g-n-exm}
g_n(x) = a_nh(c_nx),\ \ n = 1, 2, \ldots
\end{equation}
Examples: $h(x) = 1_{[0, 1]}(x)$, $h(x) = 1_{[-1, 0]}(x)$, are taken from the companion paper \cite{MyOwn7}. 
\end{exm}

\begin{exm} We can also weaken the assumptions on $h$ from Example 1:  Assume $h$ is nonincreasing on $[c_*, \infty)$ for some $c_* > 0$, and 
$$
\int_{-\infty}^{\infty}h(y)\,\md y < \infty,\ \frac{a_n}{c_n} \to \infty,\ a_nh(c_n\delta) \to 0\ \mbox{for all}\ \delta > 0.
$$
Then it is straightforward to prove that $(g_n)_{n \ge 1}$ from~\eqref{eq:g-n-exm} is, in fact, a singular sequence. An example: 
$$
h(x) = 1_{[0, \infty)}(x)e^{-x},\ a_n := n^2,\ c_n := n;\ \mbox{then}\ g_n(x) = n^2e^{-nx}.
$$
\end{exm}

The following theorem is the second main result of this paper. 

\begin{thm} Under Assumptions 1, 2, 3, 4, and conditions (b), (c), (d), (e) of Theorem~\ref{thm:1}, suppose that for every compact subset $\CK \subseteq \CU$, the sequence $(\cm_{\CK, n})_{n \ge 1}$ has a spike at zero. Then the sequence $(X_n)_{n \ge 1}$ is asymptotically in $\ol{D}$, and therefore $X_n \Ra Z$ in $C([0, T], \BR^d)$ for every $T > 0$.
\label{thm:2}
\end{thm}

Finally, let us provide an example of such sequence. Take a sequence $(g_n)_{n \ge 1}$ of functions $\BR \to \BR_+$.  Define the following sequence $(f_n)_{n \ge 1}$:
\begin{equation}
\label{eq:f-n-example}
f_n : \mathcal U \to \BR^d,\ \  f_n(x) :=
\begin{cases}
g_n(\phi(x))r(y(x)),\ x \in \CU_0;\\
0,\ x \in \CU\setminus\CU_0.
\end{cases}
\end{equation}

%The next result gives an example of such sequence.

\begin{thm} Assume the sequence $(g_{n})_{n \ge 1}$, $i = 1, \ldots, m$, is singular. Under Assumptions 1, 2, 3, 4, and conditions (d), (e) from Theorem~\ref{thm:1}, as $n \to \infty$, 
$$
X_n \Ra Z\ \mbox{in}\ C([0, T], \BR^d)\ \mbox{for every}\ T > 0.
$$
\label{thm:3}
\end{thm}

\section{Proof of Theorem~\ref{thm:1}} 

\subsection{Overview of the proof} We split the proof into a sequence of lemmata. In the following subsections, we prove each of these lemmata. Take a compact subset $\CK \subseteq \CU$. Let $\CO := \Int\CK$. 

\begin{lemma} There exists a $\de_{\CK} > 0$ such that $\CK(\de_{\CK}) \subseteq \CU_0$. 
\label{lemma:aux}
\end{lemma}

Fix time horizon $T > 0$, and an $n = 1, 2, \ldots$ For $t \ge 0$, define
$$
V_n(t) := z_n + \int_0^{t}b(X_n(s))\,\md s + \int_0^t\si_n(X_n(s))\,\md W_n(s),
$$
\begin{equation}
\label{eq:L-l-n}
L_n(t) := \int_0^tf_n(X_n(s))\,\md s,\ \ l_n(t) := \int_0^t\norm{f_n(X_n(s))}\,\!\md s. 
\end{equation}
Then we can represent the process $X_n$ as follows: 
\begin{equation}
\label{eq:X-n-sum}
X_n(t) = V_n(t) + L_n(t),\ \ t \ge 0.
\end{equation}
Consider the stopped process $X_n^{\CO}$, defined in~\eqref{eq:stopped}. 

\begin{lemma} The sequence $\left(V_n\left(\cdot\wedge\tau_{\CO}^{X_n}\right)\right)_{n \ge 1}$ is tight in $C([0, T], \BR^d)$. 
\label{lemma:V-n-tight}
\end{lemma}

\begin{lemma} The sequence $\left(l_n\left(\cdot\wedge\tau_{\CO}^{X_n}\right)\right)_{n \ge 1}$ is tight in $C[0, T]$. 
\label{lemma:l-n-tight}
\end{lemma}

Compare $L_n$ and $l_n$ from~\eqref{eq:L-l-n}. From Lemmata~\ref{lemma:l-n-tight} and~\ref{lemma:bounded-tight}, the sequence $\bigl(L_n\bigl(\cdot\wedge\tau_{\CO}^{X_n}\bigr)\bigr)_{n \ge 1}$ is tight in $C([0, T], \BR^d)$. In addition, all $W_n$ have the same law (of a $d$-dimensional Brownian motion) in $C([0, T], \BR^d)$. By Lemma~\ref{lemma:stopping-tight}, the sequence $\bigl(W_n\bigl(\cdot\wedge\tau^{X_n}_{\CO}\bigr)\bigr)_{n \ge 1}$ is tight in $C([0, T], \BR^d)$. From~\eqref{eq:X-n-sum}, we have:
\begin{equation}
\label{eq:stopped-sum}
X_n^{\CO}(t) = X_n\bigl(t\wedge\tau^{X_n}_{\CO}\bigr) = V_n\bigl(t\wedge\tau^{X_n}_{\CO}\bigr) + L_n\bigl(t\wedge\tau^{X_n}_{\CO}\bigr).
\end{equation}
Therefore, the following sequence is tight in $C([0, T], \BR^{4d+1})$:
\begin{equation}
\label{eq:sequence-big}
\bigl(\bigl(X_n^{\CO}, L_n\bigl(\cdot\wedge\tau^{X_n}_{\CO}\bigr), l_n\bigl(\cdot\wedge\tau_{\CO}^{X_n}\bigr), V_n\bigl(\cdot\wedge\tau_{\CO}^{X_n}\bigr), W_n\bigl(\cdot\wedge\tau^{X_n}_{\CO}\bigr)\bigr)\bigr)_{n \ge 1}.
\end{equation}
Take $\left(\ol{Z}, \ol{L}, \ol{l}, \ol{V}, \ol{W}\right)$, a weak limit point of the sequence~\eqref{eq:sequence-big} in $C([0, T], \BR^{4d+1})$. Without loss of generality, for simplicity of notation we can assume that the whole sequence~\eqref{eq:sequence-big} converges weakly to this limit point. By the Skorohod representation theorem, after changing the probability space, we can take a.s. uniform convergence on $[0, T]$, instead of weak convergence. We can assume the filtration $(\CF_t)_{t \ge 0}$ is generated by all these processes. Taking the limit in~\eqref{eq:stopped-sum}, we have:
$$
\ol{Z}(t) = \ol{V}(t) + \ol{L}(t),\ \ t \in [0, T].
$$  

\begin{lemma}
The process $\ol{L} = (\ol{L}(t), 0 \le t \le T\wedge\tau^{\ol{Z}}_{\CO})$ can be represented as
$$
\ol{L}(t) = \int_0^t\frac{r(\ol{Z}(s))}{\norm{r(\ol{Z}(s))}}\, \md\ol{l}(s),\ \ \mbox{for}\ \ t \le \tau^{\ol{Z}}_{\CO}\wedge T.
$$
\label{lemma:L-l}
\end{lemma}

\begin{lemma} The process $\ol{V} = (\ol{V}(t), 0 \le t \le T\wedge\tau^{\ol{Z}}_{\CO})$ can be represented as
$$
\ol{V}(t) = z_0 + \int_0^tb(\ol{Z}(s))\,\md s + \int_0^t\si(\ol{Z}(s))\,\md\ol{W}(s),\ \ \mbox{for}\ \ t \le \tau^{\ol{Z}}_{\CO}\wedge T.
$$
\label{lemma:V-view}
\end{lemma}

\begin{lemma} The process $\ol{Z}$ takes values only in $\ol{D}$ a.s.:
$\MP\Bigl(\ol{Z}(t)\in \ol{D}\ \ \forall t \in [0, T\wedge\tau^{\ol{Z}}_{\CO}]\Bigr) = 1$. 
\label{lemma:Z-in-D}
\end{lemma}

\begin{lemma} The process $\ol{l} = (\ol{l}(t), 0 \le t \le T\wedge\tau^{\ol{Z}}_{\CO})$ has a.s. nondecreasing trajectories, $\ol{l}(0) = 0$, and  $\ol{l}$ can grow only when $\ol{Z}(t) \in \pa D$; that is, 
$$
\int_0^{T\wedge\tau^{\ol{Z}}_{\CO}}1(\ol{Z}(t) \in D)\,\md\ol{l}(t) = 0.
$$
\label{lemma:prop-l}
\end{lemma}

Combining Lemmata~\ref{lemma:L-l}, ~\ref{lemma:V-view}, ~\ref{lemma:Z-in-D}, ~\ref{lemma:prop-l}, we get that the process $\ol{Z}$ stopped at exiting $\CO$ behaves as the reflected diffusion $Z$ stopped when exiting $\CO$: that is, $\ol{Z}^{\CO} \stackrel{d}{=} Z^{\CO}$. Combining this with the following lemma, we complete the proof of Theorem~\ref{thm:1}. 

\begin{lemma} Assume for any compact $\CK \subseteq \CU$ with $\CO = \Int\CK$, every weak limit point $\ol{Z}$ of  $(X_n^{\CO})_{n \ge 1}$ in $C([0, T], \BR^d)$ satisfies $\ol{Z}^{\CO} \stackrel{d}{=} Z^{\CO}$. Then $X_n \Ra Z$ in $C([0, T], \BR^d)$ as $n \to \infty$.  
\label{lemma:convergence-final}
\end{lemma}

\subsection{Proof of Lemma~\ref{lemma:aux}} Assume the converse; then there exists a sequence $x_n \in \CK$ such that $\dist(x_n, \pa D) \to 0$ and $x_n \notin \CU$. Extract a convergent subsequence $(x_{n_k})_{k \ge 1}$ and let $a$ be the limit. Then $\dist(a, \pa D) = \lim\dist(x_{n_k}, \pa D) = 0$. Therefore, $a \in \pa D$. On the other hand, since $\CU$ is open, and $x_{n_k} \notin \CU$, then $a \in \BR^d\setminus\CU$. But $\pa D \subseteq \CU$. This contradiction completes the proof.

\subsection{Preliminary calculations} The function $x \mapsto r(x)/\norm{r(x)}$ is continuous on $\pa D$, and takes values on the unit sphere $\mathbb S := \{z \in \BR^d\mid\norm{z} = 1\}$ in $\BR^d$. By Tietze's extension theorem, see \cite{Folland}, we can extend it to a continuous function $\mg : \BR^d \to \BR$, which satisfies
\begin{equation}
\label{eq:def-g}
\mg(z) \equiv \frac{r(z)}{\norm{r(z)}},\ \ z \in \pa D.
\end{equation}

\begin{lemma}
\label{lemma:closedness}
Define the set $\Xi(\de, \eps) := \left\{(z, \tilde{z}) \in \tilde{\CK}(\de)\times\BR^d\mid \norm{z - \tilde{z}} < \de,\ \ \norm{f_n(z)} \ge \eps\right\}$. For every $\eps > 0$, we have the following convergence:
$$
\lim\limits_{\de \to 0}\varlimsup\limits_{n \to \infty}\sup\limits_{(z, \tilde{z}) \in \Xi(\de, \eps)}\norm{\frac{f_n(z)}{\norm{f_n(z)}} - \mg(\tilde{z})} = 0.
$$
\end{lemma}

\begin{proof} There exists an $n_0 > 0$ such that for $n \ge n_0$, if $(z, \tilde{z}) \in \Xi(\de, \eps)$, then we have: $\phi(z) < \de$. This follows from uniform convergence $f_n \to 0$ on the set $\{x \in \CK\mid \phi(x) \ge \de\}$ (which is a compact subset of $D$). In the rest of the proof, we let $n \ge n_0$. Take $(z, \tilde{z}) \in \Xi(\de, \eps)$. Then $\phi(z) > -\de$, since $z \in \tilde{\CK}(\de)$. Therefore, $\dist(z, \pa D) = |\phi(z)| < \de$. There exists a $z'  := y(z) \in \pa D$ such that $\norm{z - z'} < \de$, and $\norm{\tilde{z} - z'} \le \norm{\tilde{z} - z} + \norm{z - z'} < 2\de$. Therefore, 
\begin{align}
\label{eq:two-sups}
\begin{split}
\sup\limits_{(z, \tilde{z}) \in \Xi(\de, \eps)}\norm{\frac{f_n(z)}{\norm{f_n(z)}} - \mg(\tilde{z})} &\le S_1(n, \de) + S_2(\de) \\ & := \sup\norm{\frac{f_n(z)}{\norm{f_n(z)}} - \mg(y(z))} + 
\sup\norm{\mg(z') - \mg(\tilde{z})},
\end{split}
\end{align}
where the first supremum $S_1(n, \de)$ is taken over all $z \in \CK(\de)$ such that $\norm{f_n(z)} \ge \eps$, and the second supremum $S_2(\de)$ is taken over all $z', \tilde{z} \in\CK$ such that $\norm{z' - \tilde{z}} < 2\de$. The supremum $S_2(\de)$ is independent of $n$ and tends to zero as $\de \to 0$, because the function $\mg$ is uniformly continuous on the compact set $\CK$. By the assumption (b) of Theorem~\ref{thm:1}, $\varlimsup_{n \to \infty}\lim_{\de \to 0}S_1(n, \de) = 0$. This completes the proof.
\end{proof}

Recall $\delta_{\CK}$ from Lemma~\ref{lemma:aux}. Take a function $\psi : \BR \to \BR$ which is $C^2$, nondecreasing, strictly increasing on $[-\de_{\CK}, \de_{\CK}]$, and
$$
\psi(x) = 
\begin{cases}
x,\ \ |x| \le \de_{\CK}/2;\\
-\de_{\CK},\ x \le - \de_{\CK};\\
\de_{\CK},\ x \ge \de_{\CK}.
\end{cases}
$$
%By Lemma~\ref{lemma:aux}, the function $\phi$ is $C^2$ on a neighborhood $V$ of $\CK(\de_{\CK})$. 
The function $\psi\circ\phi$ is $C^2$ on the neighborhood $\CU_0$. In addition, the function $\psi\circ\phi$ is $C^2$ on $\{x \in \BR^d\mid \phi(x) > \de_{\CK}\}$. But 
$\CK \subseteq \{x \in \BR^d\mid \phi(x) > \de_{\CK}\} \cup \CU_0$. Therefore, the function $\psi\circ\phi$ is $C^2$ on a neighborhood of $\CK$. Apply this function to the process $X_n$ until it exits $\CK$. By It\^o's formula for $Y_n(t) \equiv \psi(\phi(X_n(t)))$, $t < \tau^{X_n}_{\CK}$, we get:
\begin{equation}
\label{eq:ito}
\md Y_n(t) = \al_n(Y_n(t))\,\md t + \md M_n(t),
\end{equation}
Here, we let
\begin{equation}
\label{eq:al-n}
\al_n(x) := \nabla(\psi\circ\phi)(x)\cdot\left[f_n(x) + b(x)\right] + \frac12\tr(A_n(x)D^2(\psi\circ\phi)(x)),
\end{equation}
and $M_n = (M_n(t), t \ge 0)$ is a certain continuous local $(\CF_t)_{t \ge 0}$-martingale, with 
\begin{align}
\label{eq:M-n}
\begin{split}
\langle M_n\rangle_t &= \int_0^{t\wedge\tau_{\CK}^{X_n}}\be^2_n(X_n(s))\,\md s,\ \ t \ge 0,\\ 
\be_n(x) &:= \left[\left(\nabla(\psi\circ\phi)(x)\right)'A_n(x)\nabla(\psi\circ\phi)(x)\right]^{1/2}. 
\end{split}
\end{align}
 From~\eqref{eq:ito} and~\eqref{eq:M-n}, for a one-dimensional Brownian motion $B_n = (B_n(s), s \ge 0)$, we get:
\begin{equation}
\label{eq:Y-n-main}
\psi\left(\phi\left(X_n^{\CK}(t)\right)\right) = \psi(\phi(z_n)) + \int_0^{t\wedge\tau^{\CK}_{X_n}}\al_n(X_n(s))\,\md s + \int_0^{t\wedge\tau^{\CK}_{X_n}}\be_n(X_n(s))\,\md B_n(s),
\end{equation}
We have: $\phi \in C^2(\CU_0)$. But $\psi\circ\phi(x) \equiv \phi(x)$ on $\CK(\de_{\CK}/2)
\subseteq\CU_0$. Therefore, $\phi$ and $\psi\circ\phi$ have the same first and second derivatives on an open set $\CO(\de_{\CK}/2) \subseteq \CK(\de_{\CK}/2)$. For $x \in \CO(\de_{\CK}/2)$, we get:
$$
\al_n(x) = \nabla\phi(x)\cdot\left[f_n(x) + b(x)\right] + \frac12\tr(A_n(x)D^2\phi(x)),\ \ \be_n(x) = \left[\left(\nabla\phi(x)\right)'A_n(x)\nabla\phi(x)\right]^{1/2}. 
$$
Let us prove some estimates for functions $\al_n$ and $\be_n$. Recall the definition of 
$\cm_{\CK, n}$ from~\eqref{eq:m-k-n}.

\begin{lemma} There exist constants $c_1, c_2 > 0$, $\de_0 \in (0, \de_{\CK}/2)$, $n_0 = 1, 2, \ldots$ such that:
$$
\al_n(x) \ge c_1\norm{f_n(x)} - c_2 \ge c_1\cm_{\CK, n}(\phi(x)) - c_2,\ \ \mbox{for}\ \ n \ge n_0,\ \ x \in \CO(\de_0). 
$$    
\label{lemma:prop-alpha}
\end{lemma}

\begin{proof} The second inequality automatically follows from the definition of 
$\cm_{\CK, n}$ from~\eqref{eq:m-k-n}. Let us show the first inequality. The matrix-valued function $D^2(\psi\circ\phi)$, as well as the vector-valued functions $b$ and $\nabla(\psi\circ\phi)$, are continuous on the compact set $\CK$, and are therefore bounded on it. In addition, $A_n \to A$ uniformly on $\CK$, and $A$ is continuous (therefore bounded) on $\CK$. Hence the sequence $(A_n)_{n \ge 1}$ of matrix-valued functions is uniformly bounded on $\CK$. By~\eqref{eq:al-n}, it suffices to prove the statement of Lemma~\ref{lemma:prop-alpha} for $\nabla(\psi\circ\phi)(x)\cdot f_n(x)$ instead of $\al_n(x)$. For $x \in \CO(\de_{\CK}/2)$,
\begin{equation}
\label{eq:f-n-dot}
\nabla(\psi\circ\phi)(x)\cdot f_n(x) = \norm{f_n(x)}\nabla\phi(x)\cdot\left(\frac{f_n(x)}{\norm{f_n(x)}} - \mg(x)\right) + \norm{f_n(x)}\nabla\phi(x)\cdot \mg(x).
\end{equation}
For $x \in \CK\cap\pa D$, the vector $\nabla\phi(x)$ has the same direction as $\fn(x)$ (the inward unit normal vector to $\pa D$ at the point $x$), because $\phi$ is the signed distance function. Also, $\mg(x) = r(x)/\norm{r(x)}$, where $r(x)\cdot\fn(x) = 1$. Therefore, $\nabla\phi(x)\cdot \mg(x) > 0$. Because $\nabla\phi\cdot \mg$ is a continuous function on $\CK\cap\pa D$, there exists a constant $c_1 > 0$ such that 
\begin{equation}
\label{eq:nabla-phi-g-1}
\nabla\phi(x)\cdot \mg(x) \ge 3c_1\ \mbox{for}\ \ x \in \CK\cap\pa D.
\end{equation} 
%Indeed, assume the converse: there exists a neighborhood $\tilde{\CU}$ and a sequence $x_n \in \CK(n^{-1})\setminus\tilde{\CU}$. By compactness of $\CK$, there exists a convergent subsequence $x_{n_k} \to x_0$. Because $\tilde{\CU}$ is open and $\CK$ is closed, we have: $x_0 \in \CK\setminus\tilde{\CU}$. Because $\dist(x_n, \pa D) \le 1/n$, we get: $\dist(x_0, \pa D) = 0$, that is, $x_0 \in \pa D$, which contradicts $\CK\cap\pa D\setminus \tilde{\CU}$. This proves that for every neighborhood $\tilde{\CU}$ of $\CK\cap\pa D$, there exists a $\de(\tilde{\CU}) > 0$ such that $\CK\left(\de(\tilde{\CU})\right) \subseteq \tilde{\CU}$.

Now, the set $\CK(\de_{\CK}/2)$ is compact, and therefore the function $\nabla\phi\cdot \mg$ is uniformly continuous on this set. Comparing it with~\eqref{eq:nabla-phi-g-1}, we get: There exists a neighborhood $\tilde{\CU}$ of $\CK\cap\pa D$ such that $\nabla\phi(x)\cdot \mg(x) \ge 2c_1$ for $x \in \tilde{\CU}\cap\CK(\de_{\CK}/2)$. There exists a $\tilde{\de} > 0$ such that $\CO(\tilde{\de}) \subseteq \CK(\tilde{\de}) \subseteq \tilde{\CU}$. (The proof of this is similar to that of Lemma~\ref{lemma:aux}.) Take $\de_1 \in (0, \tilde{\de}\wedge(\de_{\CK}/2))$. Then
\begin{equation}
\label{eq:nabla-phi-g-2}
\nabla\phi(x)\cdot \mg(x) \ge 2c_1\ \ \mbox{for}\ \ x \in \CO(\de_1).
\end{equation}
Let $\eps_0 := c_1\bigl(\max_{x \in \CK(\de_{\CK})}\norm{\nabla\phi(x)}\bigr)^{-1} > 0$. Applying Lemma~\ref{lemma:closedness} to $z = \tilde{z} = x$, $\eps := 1$, we have: there exists a $\de_0 \in (0, \de_1)$ and an $n_1$ such that for $n \ge n_1$ and $x \in \CO(\de_0)$, 
\begin{equation}
\label{eq:f-n-g-2}
\mbox{if}\ \ \norm{f_n(x)} \ge 1,\ \ \mbox{then we have:}\ \ \norm{\frac{f_n(x)}{\norm{f_n(x)}} - \mg(x)} \le \eps_0.
\end{equation}
Combining~\eqref{eq:f-n-dot}, \eqref{eq:nabla-phi-g-2}, \eqref{eq:f-n-g-2}, we get: for $x \in \CO(\de_0)$ and $n \ge n_1$, 
\begin{equation}
\label{eq:almost-there-1}
\mbox{if}\ \ \norm{f_n(x)} \ge 1,\ \ \mbox{we have:}\ \ \nabla(\psi\circ\phi)(x)\cdot f_n(x) = \nabla\phi(x)\cdot f_n(x) \ge c_1\norm{f_n(x)}.
\end{equation}
%Note that $f_n \to 0$ uniformly on $\{x \in \CK\mid\phi(x) \ge \de_0\}$ (it is a compact subset of $D$). There exists an $n_2$ such that for $n \ge n_2$ and $x \in \CK$ with $\phi(x) \ge \de_0$, we get: $\norm{f_n(x)} \le 1$. 
By Cauchy-Schwartz inequality, if $\norm{f_n(x)} \le 1$, then
\begin{equation}
\label{eq:almost-there-2}
\nabla(\psi\circ\phi)(x)\cdot f_n(x) \ge -c_2,\ \ c_2 := \max\limits_{x \in \CK}\norm{\nabla(\psi\circ\phi)(x)}.
\end{equation}
Combining~\eqref{eq:almost-there-1} and~\eqref{eq:almost-there-2}, we complete the proof. 
%for $x \in \CK$ with $\phi(x) \ge -\de_0$, and for $n \ge n_0 := n_1\vee n_2$,
%$$
%\nabla(\psi\circ\phi)(x)\cdot f_n(x) \ge \min\left(c_1\norm{f_n(x)}, -c_2\right) \ge c_1\norm{f_n(x)} - c_2.
%$$
%This completes the proof. 
\end{proof}

\begin{lemma}
\label{lemma:prop-beta}
There exist constants $c_3, c_4 > 0$ and $n_1$ such that for $n \ge n_1$, if $x \in \CO(\de_0)$, where $\de_0$ is taken from Lemma~\ref{lemma:prop-alpha}, we have: $c_3 \le \be_n(x)$, and for all $x \in \CK$, we have: $\be_n(x) \le c_4$.  
\end{lemma}

\begin{rmk} Without loss of generality, we can take $n_1$ to be the same as $n_0$ from Lemma~\ref{lemma:prop-beta}. 
\end{rmk} 

\begin{proof} From the proof of Lemma~\ref{lemma:prop-alpha}, we get: $\nabla(\psi\circ\phi)(x) \ne 0$ on $\CK' := \{x \in \CK\mid |\phi(x)| \le \de_0\}$, and $\nabla(\psi\circ\phi)$ is continuous on $\CK$. Therefore, 
\begin{equation}
\label{eq:d-3-4}
d_1 := \min\limits_{x \in \CK'}\norm{\nabla(\psi\circ\phi)(x)} > 0,\ \ d_2 := \max\limits_{x \in \CK}\norm{\nabla(\psi\circ\phi)(x)} < \infty.
\end{equation}
Note that $\CO(\de_0) \subseteq \CK'$. The matrix-valued function $A(x)$ is continuous on $\CK$, and for every $x \in \CK$, the matrix $A(x)$ is positive definite and symmetric. Also, $A_n \to A$ uniformly on $\CK$. Therefore, there exist $d_3, d_4 > 0$ and an $n_1$ such that for every $\xi \in \BR^d$, $x \in \CK$, $n \ge n_1$, 
\begin{equation}
\label{eq:d-5-6}
d_3\norm{\xi}^2 \le \xi'A_n(x)\xi \le d_4\norm{\xi}^2.
\end{equation}
To complete the proof, combine~\eqref{eq:d-3-4} and~\eqref{eq:d-5-6} and let $c_3 := d_1d_3^{1/2}$, $c_4 := d_2d_4^{1/2}$.
\end{proof}

\subsection{Proof of Lemma~\ref{lemma:V-n-tight}} For $s \le \tau^{X_n}_{\CO}$, we have:  $X_n(s) \in \CO \subseteq \CK$. Because the function $\si$ is continuous on $\CK$, and $\si_n \to \si$ uniformly on $\CK$, we have: 
$$
\sup\limits_{\substack{n \ge 1\\x \in \CO}}\norm{\si_n(x)} \le \sup\limits_{\substack{n \ge 1\\x \in \CK}}\norm{\si_n(x)} =: C_{\si}< \infty.
$$
Therefore, a.s. for all $n = 1, 2, \ldots$ and $t \in [0, T]$, we have:
$$
\norm{1\left(t \le \tau_{\CO}^{X_n}\right)\si_n(X_n(t))} \le C_{\si} < \infty.
$$
From \cite[Lemma 7.4]{MyOwn6}, the following sequence of processes is tight in $C([0, T], \BR^d)$:
$$
t \mapsto \int_0^{t\wedge\tau^{X_n}_{\CO}}\si_n(X_n(s))\,\md W_n(s) \equiv \int_0^t1\left(s \le \tau_{\CO}^{X_n}\right)\si_n\left(X_n\left(s\wedge\tau_{\CO}^{X_n}\right)\right)\,\md W_n(s),\ \ n = 1, 2, \ldots
$$
Also, the function $b$ is continuous and therefore bounded on $\CK$. Therefore, 
$$
\sup\limits_{\substack{n \ge 1\\0 \le t \le T}}\norm{b(X_n(t))1(t \le \tau^{X_n}_{\CO})} < \infty,
$$
and by the Arzela-Ascoli theorem the following sequence is tight in $C([0, T], \BR^d)$: 
$$
t \mapsto \int_0^{t\wedge\tau_{\CO}^{X_n}}b(X_n(s))\,\md s \equiv \int_0^t
1\left(s \le \tau_{\CO}^{X_n}\right)b\left(X_n\left(s\wedge\tau_{\CO}^{X_n}\right)\right)\md s,\ \ n = 1, 2, \ldots
$$
It suffices to note that $z_n \to z_0$, and to use the definition of $V_n(t)$.

\subsection{Proof of Lemma~\ref{lemma:l-n-tight}} First, let us prove a simple auxillary lemma.

\begin{lemma}
\label{lemma:phi-Lip}
The function $\phi$ satisfies $|\phi(x) - \phi(y)| \le \norm{x - y}$ for all $x, y \in \BR^d$. 
\end{lemma}

\begin{proof} It is sufficient to prove $\phi(x) - \phi(y) \le \norm{x-y}$; then $\phi(y) - \phi(x) \le \norm{x-y}$ will follow automatically, and together these will give the required statement. We consider three cases:

\smallskip

{\it Case 1.} $x, y \in \ol{D}$. Then $\phi(x) - \phi(y) = \dist(x, \pa D) - \dist(y, \pa D) \le \norm{x-y}$, because the function $\dist(\cdot, \pa D)$ is $1$-Lipschitz.

\smallskip

{\it Case 2.} $x, y \in \BR^d\setminus D$. Then similarly $\phi(x) - \phi(y) = -\dist(x, \pa D) + \dist(y, \pa D) \le \norm{x-y}$.

\smallskip

{\it Case 3.} $x \in \ol{D}$, $y \in \BR^d\setminus D$. Take a $z \in [x, y]\cap\pa D$. Then $\phi(x) = \dist(x, \pa D) \le \norm{x - z}$, and $\phi(y) = -\dist(x, \pa D) \ge -\norm{y-z}$; therefore, $\phi(x) - \phi(y) \le \norm{x-z} + \norm{y-z} = \norm{x - y}$.
\end{proof}

Apply the  function $\phi$ from~\eqref{eq:phi} to the process $X_n(t)$ for $t \le \tau^{X_n}_{\CO}\wedge T$. 

\begin{lemma} The sequence $\left(\phi\left(X_n\left(\cdot\wedge\tau^{X_n}_{\CO}\right)\right)\right)_{n \ge 1}$ is tight in $C[0, T]$. 
\label{lemma:phi-X-n-tight}
\end{lemma}

\begin{proof} Because $\phi(X_n(0)) = \phi(z_n) \to \phi(z_0)$, by the Arzela-Ascoli criterion it suffices to fix $\de, \eps > 0$ and show that
\begin{equation}
\label{eq:tight-phi-Xn}
\lim\limits_{\de\to 0}\sup\limits_{n \ge 1}\MP\left(\oa\left(\phi\left(X_n\left(\cdot\wedge\tau^{X_n}_{\CO}\right)\right), [0, T], \de\right) > \eps\right) = 0.
\end{equation}
Assume that for a certain $n \ge 1$, the following events have happened:
$$
A_n := \Bigl\{\min\limits_{0 \le t \le T}\phi\left(X_n\left(t\wedge\tau^{X_n}_{\CO}\right)\right) > -\frac{\eps}2\Bigr\},\ \ B_n(\de) := \left\{\oa\left(\phi\left(X_n\left(\cdot\wedge\tau^{X_n}_{\CO}\right)\right), [0, T], \de\right) > \eps\right\}.
$$
Because $B_n(\de)$ has happened, there exist $t'_n, t''_n \in [0, T]$ such that $|t'_n- t''_n| \le \de$, and  
$$
\phi\left(X_n\left(t'_n\wedge\tau^{X_n}_{\CO}\right)\right) - \phi\left(X_n\left(t''_n\wedge\tau^{X_n}_{\CO}\right)\right) > \eps.
$$
Because $A_n$ has happened, we have:
$$
\phi\left(X_n\left(t''_n\wedge\tau^{X_n}_{\CO}\right)\right) > -\frac{\eps}2,\ \ \mbox{and therefore}\ \ \phi\left(X_n\left(t'_n\wedge\tau^{X_n}_{\CO}\right)\right) > \frac{\eps}2.
$$
Assume without loss of generality that $t''_n < t'_n$. By the intermediate value theorem, and by continuity of $\phi(X_n(\cdot\wedge\tau^{X_n}_{\CO}))$, there exists a $t_n$ between $t'_n$ and $t''_n$ such that for $s_n := t_n\wedge\tau_{\CO}^{X_n}$, $s'_n := t'_n\wedge\tau_{\CO}^{X_n}$, 
\begin{equation}
\label{eq:event-eps-4}
\phi(X_n(s'_n)) - \phi(X_n(s_n)) = \frac{\eps}4,\ \ \mbox{and}\ \ \phi(X_n(t\wedge\tau^{X_n}_{\CO})) \ge \frac{\eps}4,\ \ \mbox{for}\ \ t \in [t_n, t'_n], 
\end{equation}
Therefore, $s_n \le s'_n$, but $s'_n - s_n \le t'_n - t_n \le \de$. By Lemma~\ref{lemma:phi-Lip} we have:
\begin{equation}
\label{eq:difference}
\norm{X_n(s'_n) - X_n(s_n)} \ge \frac{\eps}4,\ \ \mbox{and}\ \ \phi(X_n(s)) \ge \frac{\eps}4,\ \ \mbox{for}\ \ s \in [s_n, s'_n].
\end{equation}
Consider the subset $\CK' := \{z \in \CK\mid\phi(z) \ge \eps/4\}$. This is a compact subset of $D$; therefore, $f_n \to 0$ uniformly on $\CK'$. Take an $n_1$ such that 
for $n \ge n_1$, we have: $\max_{z \in \CK'}\norm{f_n(z)} \le \eps/(8T)$. From~\eqref{eq:X-n-sum} and~\eqref{eq:L-l-n}, assuming $n \ge n_1$, we get the following estimate:
\begin{align*}
&\norm{X_n(s'_n)- X_n(s_n)} \le \norm{V_n(s'_n) - V_n(s_n)} + \int_{s_n}^{s'_n}\norm{f_n(X_n(u))}\md u \\ & \le 
\oa\left(V_n\left(\cdot\wedge\tau^{V_n}_{\CK}\right), [0, T], \de\right) + T\cdot\frac{\eps}{8T}.
\end{align*}
Combining this with~\eqref{eq:event-eps-4}, we get that the following event has happened:
$$
C_n(\de) := \Bigl\{\oa\left(V_n\left(\cdot\wedge\tau^{V_n}_{\CO}\right), [0, T], \de\right) \ge \frac{\eps}8\Bigr\}.
$$
We just proved that for $n \ge n_1$, the following inclusion is true:
$A_n\cap B_n(\de) \subseteq C_n(\de)$, or, equivalently, $B_n(\de) \subseteq C_n(\de)\cup(\Omega\setminus A_n)$. Therefore, $\MP(B_n(\de)) \le \MP(C_n(\de)) + \MP(\Omega\setminus A_n)$. Now, fix a small probability $\eta > 0$. There exists an $n_2$ such that for $n \ge n_2$, we have: $\MP(\Omega\setminus A_n) < \eta$, because $(X_n)_{n \ge 1}$ is asymptotically in $D$. By Lemma~\ref{lemma:V-n-tight}, the sequence $(V_n(\cdot\wedge\tau^{X_n}_{\CO}))_{n \ge 1}$ is tight in $C([0, T], \BR^d)$. Therefore, there exists a $\de_1$ such that for $\de \in (0, \de_1)$, we have: $\MP(C_n(\de)) < \eta$ for all $n$.  Therefore, for $\de \in (0, \de_1)$ and $n \ge n_0 := n_1\vee n_2$, we have: $\MP(B_n(\de)) < 2\eta$. This completes the proof.    
\end{proof}

\begin{lemma}
\label{lemma:int-a-n-tight}
The following sequence is tight in $C[0, T]$:
$$
t \mapsto \int_0^{t\wedge\tau^{X_n}_{\CO}}\al_n(X_n(s))\,\md s,\ \ n = 1, 2, \ldots
$$
\end{lemma}

\begin{proof} By Lemmata~\ref{lemma:phi-X-n-tight} and~\ref{lemma:cont-mapping-tight}, the sequence 
$$
(\psi(\phi(X_n^{\CO}(\cdot))))_{n \ge 1}
$$
is tight in $C[0, T]$. By Lemma~\ref{lemma:prop-beta}, and \cite[Lemma 7.3]{MyOwn6}, the following sequence is tight in $C[0, T]$: 
$$
t \mapsto \int_0^{t\wedge\tau^{X_n}_{\CO}}\be_n(X_n(s))\,\md B_n(s),\ \ n = 1, 2, \ldots
$$
By continuity of $\phi$ and $\psi$, and by condition (e) of Theorem~\ref{thm:1}, we have: $\psi(\phi(z_n)) \to \psi(\phi(z_0))$. Recall~\eqref{eq:Y-n-main} and complete the proof.
\end{proof} 

\begin{lemma} The sequence $\bigl(l_n\bigl(\cdot\wedge\tau^{X_n}_{\tilde{\CO}(\de_0)}\bigr)\bigr)_{n \ge 1}$ is tight in $C([0, T], \BR^d)$.
\label{lemma:l-n-stopped-tight} 
\end{lemma}

\begin{proof} 
Note that $\tilde{\CO}(\de_0) \subseteq \CO$, and so $\tau^{X_n}_{\CO}\wedge\tau^{X_n}_{\tilde{\CO}(\de_0)} = \tau^{X_n}_{\tilde{\CO}(\de_0)}$. Therefore, by Lemmata~\ref{lemma:int-a-n-tight} and~\ref{lemma:stopping-tight}, the following sequence is tight in $C[0, T]$:
$$
t \mapsto \int_0^{t\wedge\tau^{X_n}_{\tilde{\CO}(\de_0)}}\al_n(X_n(s))\,\md s,\ \ n = 1, 2, \ldots
$$
By Lemma~\ref{lemma:prop-alpha}, for $0 \le t'' \le t' \le T$, $n \ge n_1$,
\begin{align*}
\int_{t''\wedge\tau^{X_n}_{\tilde{\CO}(\de_0)}}^{t'\wedge\tau^{X_n}_{\tilde{\CO}(\de_0)}}&\al_n(X_n(s))\,\md s + c_2(t' - t'')  \ge c_1\int_{t''\wedge\tau^{X_n}_{\tilde{\CO}(\de_0)}}^{t'\wedge\tau^{X_n}_{\tilde{\CO}(\de_0)}}\norm{f_n(X_n(s))}\md s \\ & = c_1\left[l_n\left(t'\wedge\tau^{X_n}_{\tilde{\CO}(\de_0)}\right) - l_n\left(t''\wedge\tau^{X_n}_{\tilde{\CO}(\de_0)}\right)\right] \ge 0.
\end{align*}
Applying Lemma~\ref{lemma:three-processes-tight}, we finish the proof of Lemma~\ref{lemma:l-n-stopped-tight}. 
\end{proof}

Now, let us complete the proof of Lemma~\ref{lemma:l-n-tight}. Fix a small probability $\eta > 0$. Take an open subset $\CG \subseteq C[0, T]$. Because the sequence $(X_n)_{n \ge 1}$ of processes is asymptotically in $\ol{D}$, we have:
\begin{equation}
\label{eq:E-n}
\lim\limits_{n \to \infty}\MP(E_n) = 1,\ \ \mbox{where}\ \ E_n := \bigl\{\min\limits_{0 \le t \le T}\phi\left(X_n^{\CO}(t)\right) > -\de_0\bigr\} \supseteq \bigl\{\min\limits_{0 \le t \le T}\phi\left(X_n^{\CK}(t)\right) > -\de_0\bigr\}
\end{equation}
Note that we have the following sequence of event inclusions:
\begin{equation}
\label{eq:E-n-inclusion}
E_n \subseteq \bigl\{\tau_{\tilde{\CO}(\de_0)}^{X_n} = \tau_{\CO}^{X_n}\bigr\} \subseteq \bigl\{l_n\bigl(t\wedge\tau^{X_n}_{\tilde{\CO}(\de_0)}\bigr) = l_n\bigl(t\wedge\tau_{\CO}^{X_n}\bigr),\ \forall\ t \in [0, T]\bigr\}.
\end{equation}
Because $\bigl(l_n\bigl(\cdot\wedge\tau^{X_n}_{\tilde{\CO}(\de_0)}\bigr)\bigr)_{n \ge 1}$ is tight in $C[0, T]$, every increasing sequence $(n_k)_{k \ge 1}$ of positive integers has a subsequence $(n'_k)_{k \ge 1}$ such that $l_{n'_k}\bigl(\cdot\wedge\tau^{X_{n'_k}}_{\tilde{\CO}(\de_0)}\bigr) \Ra \ol{l}$ for some random element $\ol{l}$ in $C[0, T]$. By the portmanteau theorem, see \cite[Section 1.2]{BillingsleyBook}, 
\begin{equation}
\label{eq:weak-convergence}
\varliminf\limits_{k \to \infty}\MP\bigl(l_{n'_k}\bigl(\cdot\wedge\tau^{X_{n'_k}}_{\tilde{\CO}(\de_0)}\bigr) \in \CG\bigr) \ge \MP\left(\ol{l} \in \CG\right).
\end{equation}
Because of~\eqref{eq:E-n-inclusion}, we have the following comparison of probabilities:
$$
\MP\bigl(l_{n'_k}\bigl(\cdot\wedge\tau^{X_{n'_k}}_{\CO}\bigr) \in \CG\bigr) \ge \MP\bigl(l_{n'_k}\bigl(\cdot\wedge\tau^{X_{n'_k}}_{\tilde{\CO}(\de_0)}\bigr) \in \CG\bigr) - \MP(E_n^c).
$$
Taking $\varliminf_{n \to \infty}$ and using~\eqref{eq:E-n} and~\eqref{eq:weak-convergence}, we get:
$$
\varliminf\limits_{n \to \infty}\MP\bigl(l_{n'_k}\bigl(\cdot\wedge\tau^{X_{n'_k}}_{\CO}\bigr) \in \CG\bigr) \ge \varliminf\limits_{n \to \infty}\MP\bigl(l_{n'_k}\bigl(\cdot\wedge\tau^{X_{n'_k}}_{\tilde{\CO}(\de_0)}\bigr) \in \CG\bigr)
- \lim\limits_{n \to \infty}\MP(E_n^c) \ge \MP\left(\ol{l} \in \CG\right).
$$
Applying the portmanteau theorem, we get: $l_{n'_k}(\cdot\wedge\tau^{X_{n'_k}}_{\CK}\bigr) \Ra \ol{l}$. Therefore, every subsequence of 
$\left(l_n\left(\cdot\wedge\tau^{X_n}_{\CO}\right)\right)_{n \ge 1}$ has a weakly convergent subsequence. Therefore, 
$\left(l_n\left(\cdot\wedge\tau^{X_n}_{\CO}\right)\right)_{n \ge 1}$ is tight, which completes the proof.

\subsection{Proof of Lemma~\ref{lemma:L-l}} Take a $t < \tau^{\ol{Z}}_{\CO}$. Then we get: 
\begin{align*}
L_n(t) &= \int_0^tf_n(X_n(s))\,\md s \\ & = \int_0^t\left[\frac{f_n(X_n(s))}{\norm{f_n(X_n(s))}} - \mg(\ol{Z}(s))\right]\norm{f_n(X_n(s))}\,\md s + \int_0^t\mg(\ol{Z}(s))\,\md l_n(s).
\end{align*}
Now, $l_n \to \ol{l}$ a.s. uniformly in $[0, T]$, and by Lemma~\ref{lemma:weak-like-conv} we get a.s. convergence:
\begin{align}
\label{eq:441}
\int_0^t\mg(\ol{Z}(s))\,\md l_n(s) \to \int_0^t\mg(\ol{Z}(s))\,\md\ol{l}(s),\ \ n \to \infty.
\end{align}
Let us show that as $n \to \infty$, we have a.s.:
\begin{align}
\label{eq:442}
\begin{split}
\int_0^t&\left[\frac{f_n(X_n(s))}{\norm{f_n(X_n(s))}} - \mg(\ol{Z}(s))\right]\norm{f_n(X_n(s))}\,\md s  \\ & \equiv \int_0^t\left[\frac{f_n(X_n(s))}{\norm{f_n(X_n(s))}} - \mg(\ol{Z}(s))\right]\,\md l_n(s) \to 0.
\end{split}
\end{align}
Fix an $\eps > 0$. Note that we can split the integral into two parts:
\begin{align}
\label{eq:split}
\begin{split}
\int\limits_0^t&\left[\frac{f_n(X_n(s))}{\norm{f_n(X_n(s))}} - \mg(\ol{Z}(s))\right]\md l_n(s) \\ & = \left(\int\limits_{E_n(\eps)} + \int\limits_{F_n(\eps)}\right)\left[\frac{f_n(X_n(s))}{\norm{f_n(X_n(s))}} - \mg(\ol{Z}(s))\right]\md l_n(s).
\end{split}
\end{align}
where we define the following (random) subsets of $[0, t]$: 
$$
E_n(\eps) := \{s \in [0, t]:\ \norm{f_n(X_n(s))} < \eps\},\ \ F_n(\eps) := \{s \in [0, t]:\ \norm{f_n(X_n(s))} \ge \eps\}.
$$
One of these integrals from~\eqref{eq:split} can be estimated easily:
\begin{equation}
\label{eq:452}
\int\limits_{E_n(\eps)}\left[\frac{f_n(X_n(s))}{\norm{f_n(X_n(s))}} - \mg(\ol{Z}(s))\right]\,\md l_n(s) \le (1 + 1)\cdot\eps\cdot t \le 2\eps T.
\end{equation}
Let us estimate the other integral from~\eqref{eq:split}, over $F(\eps)$. By Lemma~\ref{lemma:closedness}, for every $\de > 0$ there exists an $n_{\de}$ such that for $n \ge n_{\de}$, we have: 
\begin{equation}
\label{eq:conv-X-n-Z}
\max\limits_{0 \le s \le T}\norm{X_n\left(s\wedge\tau^{X_n}_{\CO}\right) - \ol{Z}(s)} \le \de.
\end{equation}
From~\eqref{eq:conv-X-n-Z}, we have the following estimation:
\begin{equation}
\label{eq:complicated-estimate}
\int\limits_{F_n(\eps)}\left[\frac{f_n(X_n(s))}{\norm{f_n(X_n(s))}} - \mg(\ol{Z}(s))\right]\md l_n(s) \le l_n(T)\cdot\sup\limits_{(z, \tilde{z}) \in \Xi(\de, \eps)}\norm{\frac{f_n(z)}{\norm{f_n(z)}} - \mg(\tilde{z})}.
\end{equation}
Taking $\lim_{\de \to 0}\varlimsup_{n \to \infty}$ in~\eqref{eq:complicated-estimate}, applying Lemma~\ref{lemma:closedness} and observing that $l_n(T) \to \ol{l}(T)$ a.s., we get:
\begin{equation}
\label{eq:453}
\lim\limits_{n \to \infty}\int\limits_{F_n(\eps)}\left[\frac{f_n(X_n(s))}{\norm{f_n(X_n(s))}} - \mg(\ol{Z}(s))\right]\md l_n(s) = 0.
\end{equation}
Taking $\varlimsup_{n \to \infty}$ in~\eqref{eq:split} and applying~\eqref{eq:452} and~\eqref{eq:453}, we get:
$$
\varlimsup\limits_{n \to \infty}\int\limits_0^t\left[\frac{f_n(X_n(s))}{\norm{f_n(X_n(s))}} - \mg(\ol{Z}(s))\right]\md l_n(s) \le 2T\eps.
$$
Since $\eps > 0$ is arbitrary, this proves~\eqref{eq:442}. Combining~\eqref{eq:441} and~\eqref{eq:442} and noting that $\mg(z) = r(z)/\norm{r(z)}$ for $z \in \pa D$,  we complete the proof. 

\subsection{Proof of Lemma~\ref{lemma:V-view}} Uniformly on $[0, T]$, we have:
$X_n\left(\cdot\wedge\tau_{\CO}^{X_n}\right) \equiv X_n^{\CO}(\cdot) \to \ol{Z}(\cdot)$. By definition of $\tau^{X_n}_{\CO}$, we have: $X_n\left(t\wedge\tau_{\CO}^{X_n}\right) \in \CK$ for all $t \in [0, T]$. Because the set $\CK \supseteq \CO$ is closed, $\ol{Z}(t) \in \CK$ for $t \in [0, T]$. Take a $t < \tau^{\ol{Z}}_{\CO}$; then for large enough $n$, we have: $t < \tau^{X_n}_{\CO}$. (Indeed, otherwise, there exists a subsequence $(n_k)_{k \ge 1}$ such that $t_k := \tau_{\CO}^{X_{n_k}} \le t$. The process $X_{n_k}^{\CO}$ stops earlier than $t$: it stops at $t_k$. Therefore, $X_{n_k}(t) = X_{n_k}(t_k) \in \BR^d\setminus\CO$. Taking $k \to \infty$, we have: $\ol{Z}(t) \in \BR^d\setminus\CO$. This contradicts the assumption that $t < \tau^{\ol{Z}}_{\CO}$, because then $\ol{Z}(t) \in \CO$.) By Assumption~\ref{asmp:continuous}, the function $b$ is continuous and therefore bounded on $\CK$. Therefore, we have the following bounded convergence: $b(X_n(s)) \to b(\ol{Z}(s))$ for every $s \in [0, t]$. Applying Lebesgue dominated convergence theorem, we get:
\begin{equation}
\label{eq:b-conv}
\int_0^tb(X_n(s))\,\md s \to \int_0^tb(\ol{Z}(s))\,\md s,\ \ n \to \infty.
\end{equation}
Also, $\si_n \to \si$ uniformly on $\CK$, and $\si$ is continuous (therefore bounded) on $\CK$. Therefore, this convergence is also bounded. By Lemma~\ref{lemma:conv-of-ito}, we have: 
\begin{equation}
\label{eq:sigma-conv}
\int_0^t\si_n(X_n(s))\,\md W_n(s) \to \int_0^t\si(\ol{Z}(s))\,\md\ol{W}(s)\ \ \mbox{in}\ \ L^2.
\end{equation}
The convergence~\eqref{eq:sigma-conv} is in $L^2$, and therefore in probability. Extract an a.s. convergent subsequence:
\begin{equation}
\label{eq:si-conv}
\int_0^t\si_{n_k}(X_{n_k}(s))\,\md W_{n_k}(s) \to \int_0^t\si(\ol{Z}(s))\,\md\ol{W}(s),\ \ k \to \infty.
\end{equation}
In addition, we have the following convergence a.s.:
\begin{equation}
\label{eq:V-conv}
V_{n_k}(t) \equiv z_{n_k} + \int_0^tb(X_{n_k}(s))\,\md s + \int_0^t\si_{n_k}(X_{n_k}(s))\,\md W_{n_k}(s) \to \ol{V}(t),\ \ k \to \infty.
\end{equation}
Combining~\eqref{eq:b-conv},~\eqref{eq:si-conv},~\eqref{eq:V-conv} with $z_{n_k} \to z_0$, we complete the proof. 

\subsection{Proof of Lemma~\ref{lemma:Z-in-D}} We have a.s. uniform (and therefore weak) convergence $X_n^{\CO} \to \ol{Z}$. Fix an $\eta > 0$. Take the following open subset of $C([0, T], \BR^d)$: 
$$
\CG(\eta) := \Bigl\{x \in C([0, T], \BR^d)\mid \min\limits_{0 \le t \le T}\phi(x(t)) > -\eta\Bigr\}.
$$
By the portmanteau theorem, see \cite[Section 1.2]{BillingsleyBook}, we have:
\begin{equation}
\label{eq:142142}
\MP(\ol{Z} \in \CG(\eta)) \ge \varliminf\limits_{n \to \infty}\MP\left(X_n^{\CO} \in \CG(\eta)\right).
\end{equation}
Combine~\eqref{eq:142142} with the fact that $(X_n)_{n \ge 1}$ is asymptotically in $\ol{D}$. We get: $\MP(\ol{Z} \in \CG(\eta)) = 1$. This is true for every $\eta > 0$. But $\cap_{n \ge 1}\CG(1/n) = \{x \in C([0, T], \BR^d)\mid x(t) \in \ol{D}\}$. This completes the proof. 

\subsection{Proof of Lemma~\ref{lemma:prop-l}} That $\ol{l}$ is continuous, nondecreasing, and $\ol{l}(0) = 0$ follows from uniform convergence $l_n \to \ol{l}$ on $[0, T]$ (because each $l_n$ also has these properties). Now, let us prove that $\ol{l}$ can grow only when $\ol{Z} \in \pa D$. Assume that $\ol{Z}(s)$ is away from $\pa D$ for $s \in [s_1, s_2]$. This is equivalent to $\phi(Z(s)) > 0$ for these $s$. Because $\phi(Z(\cdot))$ is continuous, we have: $\phi(\ol{Z}(s)) \ge \eta > 0$ for $s \in [s_1, s_2] \subseteq [0, T]$. Let us show that $\ol{l}(s_1) = \ol{l}(s_2)$. By the uniform convergence $X_{n}^{\CO} \to \ol{Z}$, there exists an $n_0$ such that
$$
\norm{X_{n}^{\CO}(s) - \ol{Z}(s)} \le \frac{\eta}2\ \ \mbox{for}\ \ s \in [0, T],\ n \ge n_0.
$$
Therefore, for $s \in [s_1, s_2]$ and $n \ge n_0$, we have: 
$$
\phi(X^{\CO}_{n}(s)) \ge \phi(\ol{Z}(s)) - \norm{X_{n}^{\CO}(s) - \ol{Z}(s)} \ge \eta - \frac{\eta}2 = \frac{\eta}2.
$$
But on the compact set $\CK' := \{x \in \CK\mid \phi(x) \ge \eta/2\}$ we have:
$\norm{f_{n}} \to 0$ uniformly. Therefore, 
$$
0 \le \ol{l}(s_2) - \ol{l}(s_1) = \lim\limits_{n \to \infty}\left[l_{n}(s_2) - l_{n}(s_1)\right] \le (s_2 - s_1)\cdot\sup\limits_{\CK'}\norm{f_{n}} \to 0.
$$
This completes the proof that $\ol{l}(s_1) = \ol{l}(s_2)$.

\subsection{Proof of Lemma~\ref{lemma:convergence-final}} Fix any increasing sequence $(n_k)_{k \ge 1}$ of positive integers. Fix a certain time horizon $T > 0$ and a small probability $\eta > 0$. Define the following sequence of compact subsets of $\CU$:
$$
\CK_m := \ol{\CO}_m,\ \ \CO_m := \{x \in \CU\mid \dist(x, \pa\CU) > m^{-1},\ \ \norm{x} < m\},\ \ m = 1, 2, \ldots
$$
Then $\Int\CK_m \supseteq \CO_m$, and so $\Int\CK_m \uparrow \CU$ as $m \to \infty$. Therefore, for some large enough $m$,   
\begin{equation}
\label{eq:localize-00}
\MP\left(Z(t) \in \Int\CK_m\ \ \mbox{for all}\ \ t \in [0, T]\right) > 1 - \eta.
\end{equation}
There exists a subsequence $(n'_k)_{k \ge 1}$ such that $X^{\CO_m}_{n_k} \Ra \ol{Z}$, $k \to \infty$. By our assumptions, this limit point $\ol{Z}$ behaves as $Z$ until it exits $\CO_m$. Define the following subset of $C([0, T], \BR^d)$: 
$\CA = \{h \in C([0, T], \BR^d)\mid h(t) \in \CO_m\ \ \mbox{for all}\ \ t \in [0, T]\}$. 
From~\eqref{eq:localize-00}, we get:
\begin{equation}
\label{eq:localize-01}
\MP\left(\ol{Z} \in \CA\right) = \MP\left(Z \in \CA\right) > 1 - \eta.
\end{equation}
From the definition of the set $\CA$, for any random process $f = (f(t), 0 \le t \le T)$, we have: 
\begin{equation}
\label{eq:localize-02}
\{f \in \CA\} = \{f^{\CO_m} \in \CA\}.
\end{equation}
Since $\CA$ is open, by the portmanteau theorem, see \cite[Section 1.2]{BillingsleyBook}, and by~\eqref{eq:localize-01}, ~\eqref{eq:localize-02}, we have:
\begin{equation}
\label{eq:localize-03}
\varliminf\limits_{k \to \infty}\MP\left(X^{\CO_m}_{n'_k} \in \CA\right) \equiv \varliminf\limits_{k \to \infty}\MP\left(X_{n'_k} \in \CA\right) \ge \MP(\ol{Z} \in \CA) > 1 - \eta.
\end{equation}
Therefore, there exists a certain $k_0$ such that for $k > k_0$, $\MP\left(X_{n'_k}(t) \in \CO_m,\ \ t \in [0, T] \right) > 1 - \eta$. For any open set $\CG \subseteq C([0, T], \BR^d)$, by definition of $\CA$, we have: 
$$
\MP(X_n \in \CG) \ge \MP(X_n \in \CG\cap\CA) = \MP(X^{\CO_m}_n \in \CG\cap \CA).
$$
The set $\CG\cap\CA$ is also open in $C([0, T], \BR^d)$. Applying the portmanteau theorem again, we get:
\begin{align*}
\varliminf\limits_{k \to \infty}\MP(X_{n'_k} \in \CG) &\ge \varliminf\limits_{k \to \infty}\MP(X^{\CO_m}_{n'_k} \in \CG\cap\CA) \ge \MP(\ol{Z}^{\CO_m} \in \CG\cap\CA) = \MP(Z^{\CO_m} \in \CG\cap\CA) \\ &= \MP(Z \in \CG\cap\CA) \ge \MP(Z \in \CG) - \MP(Z \notin \CA) \ge \MP(Z \in \CG) - \eta.
\end{align*}
This seems to complete the proof, because $\eta > 0$ can be taken to be arbitrarily small. However, there is a difficulty: The subsequence $(n'_k)_{k \ge 1}$ depends on $\eta$. To bypass this, let us use a standard diagonal argument: let $\eta_m := 1/m$ and construct a sequence $(n^{(m)}_k)_{k \ge 1},\ m = 1, 2, \ldots$ of inserted subsequences with the following property: $(n^{(m+1)}_k)_{k \ge 1}$ is a subsequence of $(n^{(m)}_k)_{k \ge 1}$. For every $m$,  
$$
\varliminf\limits_{k \to \infty}\MP(X_{n^{(m)}_k} \in \CG) \ge \MP(Z \in \CG) - \frac1m.
$$
Take the diagonal subsequence $\ol{n}_k := (n_k^{(k)})_{k \ge 1}$; then 
\begin{equation}
\label{eq:localize-04}
\varliminf\limits_{k \to \infty}\MP(X_{\ol{n}_k} \in \CG) \ge \MP(Z \in \CG).
\end{equation}
We proved~\eqref{eq:localize-04} for every open set $\CG$ in $C([0, T], \BR^d)$. Thus, by the portmanteau theorem, $X_{\ol{n}_k} \Ra Z$. To summarize, we now proved that every increasing sequence $(n_k)_{k \ge 1}$ of positive integers has a subsequence $(\ol{n}_k)_{k \ge 1}$ such that $X_{\ol{n}_k} \Ra Z$ in $C([0, T], \BR^d)$. Thus, $X_n \Ra Z$ in $C([0, T], \BR^d)$. 

\section{Proof of Theorem~\ref{thm:2}}

\subsection{The main part of the proof} We use the notation from Section 4. First, in this subsection, we state a technical lemma, and apply it to the sequence $(\psi(\phi(X^{\CO}_n)))_{n \ge 1}$. This completes the proof of Theorem~\ref{thm:2}. The rest of this section is devoted to the proof of this technical lemma. 

\begin{lemma} Fix $T > 0$. Take a sequence $(Y_n)_{n \ge 1}$ of real-valued continuous adapted processes $Y_n = (Y_n(t), 0 \le t \le T)$. In addition, for every $n = 1, 2, \ldots$ take a real number $y_n > 0$, an $(\CF_t)_{t \ge 0}$-stopping time $\tau_n \le T$, an $(\CF_t)_{t \ge 0}$-Brownian motion $B_n = (B_n(t), 0 \le t \le T)$, and real-valued adapted processes $\ga_n = (\ga_n(t), 0 \le t \le T)$ and $\rho_n = (\rho_n(t), 0 \le t \le T)$. Assume that for each $n = 1, 2, \ldots$ the process $Y_n$ satisfies the following equation:
$$
\md Y_n(t) = \ga_n(t)\md t + \rho_n(t)\md B_n(t),\ \ t \le \tau_n,\ \ \mbox{and}\ \ Y_n(0) = y_n.
$$
Assume that, almost surely, for some positive constants $\eps_0, C_1, C_2, C_3, C_4, C_5$, we have:

\smallskip

(a) there exists $y_0 \in (0, \eps_0)$ such that for all $n = 1, 2, \ldots$ we have: $y_n \in [y_0, \eps_0]$;

\smallskip

(b) for $t \le \tau_n$, if $|Y_n(t)| \le \eps_0$, then $0 < C_1 \le |\rho_n(t)| \le C_2 < \infty$;

\smallskip

(c) for $t \le \tau_n$, if $|Y_n(t)| \le \eps_0$, then $\ga_n(t) \ge \ol{\ga}_n(Y_n(t))$ for some function $\ol{\ga}_n : [-\eps_0, \eps_0] \to \BR$;

\smallskip

(d) $\ol{\ga}_n(s) \ge -C_3$ for all $s \in [-\eps_0, \eps_0]$ and $n = 1, 2,\ldots$;

\smallskip

(e) for every $\eps \in (0, \eps_0)$, we have:
$$
\lim\limits_{n \to \infty}\int_{-\eps}^{\eps}\ol{\ga}_n(s)\md s = \infty;
$$

\smallskip

(f) for every $\eps \in (0, \eps_0)$, we have: 
$$
\varlimsup\limits_{n \to \infty}\sup\limits_{s \in [\eps, \eps_0]}\ol{\ga}_n(s) \le C_4;
$$

\smallskip

(g) there exists an $\eps_1 \in (0, \eps_0]$ such that if $t \le \tau_n$ and $Y_n(t) \in (\eps_1, \eps_0)$, then $\ga_n(t) \le C_5$.

\smallskip

Then for every $\de > 0$, we have:
$$
\lim\limits_{n \to \infty}\MP\Bigl(\min\limits_{0 \le t \le \tau_n}Y_n(t) > -\de\Bigr) = 1.
$$
\label{lemma:master}
\end{lemma}

%In the companion paper \cite{MyOwn7}, which deals with a similar problem as in this paper, but in a one-dimensional setting: approximation of reflected diffusions on the positive half-line. The key step in the proof was Lemma 3.2, which proved that it is unlikely that the approximating process hits any negative level $-\de$ in time $T > 0$. The corresponding part of this article is Theorem~\ref{thm:2}, which proves the same for the process $\phi(X_n)$. The techniques of the proof are similar, but there are important differences: $\phi(X_n)$ is not a solution of an SDE; this is a general It\^o process. This is the reason why the proof becomes much more compicated. We separately state the following lemma about a sequence $(Y_n)_{n \ge 1}$ of general real-valued It\^o processes, and then apply it to $Y_n = (\psi\circ\phi)(X_n)$. Fix time $T > 0$. 

Apply Lemma~\ref{lemma:master} to the equation~\eqref{eq:Y-n-main}, with $c_1, c_2$ taken from Lemma~\ref{lemma:prop-alpha}, with 
$$
\eps_0 := \de_0\wedge(\eps_{\CK}/3),\ \ \mbox{any}\ \ \eps_1 \in (0, \eps_0),\ \ 
y_0 := \psi(\phi(z_0))/2,
$$
$$
\ga_n(t) := \al_n(X_n(t)),\ \ \rho_n(t) := \be_n(X_n(t)),\ \ Y_n(t) := (\psi\circ\phi)(X_n(t)), 
$$
$$
\tau_n := \tau_{\CO}^{X_n}\wedge T,\ \ \ol{\ga}_n(s) := c_1\cm_{\CO, n}(s) - c_2,\ \ y_n := \psi(\phi(z_n))\ \ \mbox{for}\ \ n = 1, 2, \ldots
$$
Let us verify assumptions (a) - (g) of Lemma~\ref{lemma:master}. 

\smallskip

{\bf Proof of (a):} By assumption (e) of Theorem~\ref{thm:1} and continuity of functions $\psi$ and $\phi$, we have: $\psi(\phi(z_n)) \to \psi(\phi(z_0))$. But $z_0 \in D$ by Assumption~\ref{asmp:continuous}, and therefore $\phi(z_0) > 0$. From the properties of the function $\psi$ it follows that $\psi(\phi(z_0)) > 0$. Therefore, $\psi(\phi(z_n)) \ge y_0 = \psi(\phi(z_0))/2$ for large enough $n$. Without loss of generality, we can assume this for all $n$. 

\smallskip

{\bf Proof of (b):} Assume $|Y_n(t)| \le \eps_0$, and $t \le \tau^{X_n}_{\CO}\wedge T$. Rewrite this as $|\psi(\phi(X_n(t)))| \le \eps_0$. From the construction of the function $\psi$, we get: if $|\psi(s)| \le \de_{\CK}/2$, then $\psi(s) = s$. Therefore, $|\phi(X_n(t))| \le \eps_0$. But $\eps_0 \le \de_0 < \de_{\CK}/2$ by Lemma~\ref{lemma:prop-alpha}. The rest follows from Lemma~\ref{lemma:prop-beta} with $C_1 := c_3$ and $C_2 := c_4$. 

\smallskip

{\bf Proof of (c):} Follows from Lemma~\ref{lemma:prop-alpha} similarly to the proof of (b).

\smallskip

{\bf Proof of (d):} This is trivial, with $C_3 := c_2$. 

\smallskip

{\bf Proof of (e):} Follows from the fact that $(\cm_{\CK, n})_{n \ge 1}$, and therefore $(\cm_{\CO, n})_{n \ge 1}$, has a spike at zero. 

\smallskip

{\bf Proof of (f):} Uniformly on the set $\{z \in \CK\mid \eps_1 \le \phi(z) \le \eps_0\}$ (which is a compact subset of $D$), we have: $f_n \to 0$. Therefore, $\cm_{\CK, n}(s) \to 0$ uniformly for $s \in [\eps, \eps_0]$, and we can take $C_4$ to be any positive number. 

\smallskip

{\bf Proof of (g):} From the proof of (f), it follows that $f_n \to 0$ uniformly on the set 
$\CK'$. Moreover, as shown in the proof of Lemma~\ref{lemma:prop-alpha}, $\psi\circ\phi \equiv \phi$ on $\CK(\de_{\CK}/2)$, and therefore on a compact set $\{z \in \CK\mid |\phi(z)| \le \de_{\CK}/3\}$. 
The function $\phi$ is $C^2$ on an open set $\CU_0 \supseteq \CK(\de_{\CK}/2)$. Therefore, $\nabla\phi$ and $D^2\phi$ are bounded on $\{z \in \CK\mid |\phi(z)| \le \de_{\CK}/3\}$. Next, on an open set 
$$
\CO' := \{z \in \CO\mid \eps_1 < \phi(z) < \eps_0\} \subseteq \{z \in \CK\mid |\phi(z)| \le \de_{\CK}/3\},
$$
we have: $\nabla(\psi\circ\phi) \equiv \nabla\phi$, and $D^2(\psi\circ\phi) \equiv D^2\phi$. Therefore, these functions are also bounded on $\CO'$. The function $b$ is continuous and therefore bounded on $\CK'$. Finally, the sequence $(A_n)_{n \ge 1}$ is uniformly bounded on $\CK'$, because $A_n \to A$ uniformly on $\CK'$, and $A$ is continuous and therefore bounded on $\CK'$. The rest follows from~\eqref{eq:al-n}.

% follows from $z_n \to z_0$ and $z_0 \in D$, therefore $\phi(z_0) > 0$. Assumption (b) follows from
%Lemma~\ref{lemma:prop-beta}. Assumptions (c)  and (d) follow from Lemma~\ref{lemma:prop-alpha}, with $C_3 := c_2$. Assumption (e) follows from the condition that $(\cm_{\CK, n})_{n \ge 1}$ has a spike at zero. Assumption (f) follows from $f_n \to 0$ uniformly on the set $\{z \in \CK\mid \eps \le \phi(z) \le \eps_0\}$. Finally, Assumption (g) follows from fact that 
%\begin{equation}
%\label{eq:bounded-al}
%\varlimsup\limits_{n \to \infty}\sup\limits_{\substack{x \in \CK\\ \eps_1 \le  \phi(x) \le \eps_0}}|\al_n(x)| < \infty.
%\end{equation}
%Let us show~\eqref{eq:bounded-al}. Indeed, the set $\CK'' := \{x \in \CK\mid  \eps_1 \le  \phi(x) \le \eps_0\}$ is a compact subset of $D$; therefore, $f_n \to 0$ uniformly on $\CK''$. The rest follows from~\eqref{eq:al-n}, boundedness of the functions $\nabla\phi$, $D^2\phi$, $b$ and uniform boundedness of $(A_n)_{n \ge 1}$ on $\CK''$. 

\subsection{Informal description of the proof of Lemma~\ref{lemma:master}} The proof resembles the one of \cite[Lemma 3.2]{MyOwn7}. If $Y_n$ were solutions of SDEs, we could apply scale functions to them and then proceed as in \cite[Lemma 3.2]{MyOwn7}. However, $Y_n$ are not solutions of SDEs, but general It\^o processes. We apply time-change to $Y_n$ to convert the martingale part into the Brownian motion, at least as long as the process $Y_n$ is in the interval $[-\eps_0, \eps_0]$. Then we compare this time-changed process $Y_n$ with a solution $\ol{Y}_n$ of an SDE:
$$
\md\ol{Y}_n(t) = \psi_n{\ga}_n(\ol{Y}_n(t))\,\md t + \md W_n(t),
$$
where $W_n$ is a Brownian motion, and $\psi_n$ is a suitable function. In Lemma~\ref{lemma:aux-important}, we apply the scale function to this new process $\ol{Y}_n$, and show that, at least for large enough $n$, it is unlikely to hit the level $-\de$ before hitting the level $\eps_0$. Using the Markov property, we show that $\ol{Y}_n$ is unlikely to hit the level $-\de$ before hitting the level $\eps_0$ a fixed number $j$ of times (the hitting of $\eps_0$ counts as new if the process $\ol{Y}_n$ has returned to the level $y_0 > 0$ before the next hitting of $\eps_0$). But it turns out that every return of $\ol{Y}_n$ from $y_0$ to $\eps_0$ takes a substantial amount of time. Therefore, for some large enough $j$ it takes a lot of time to return $j$ times from $y_0$ to $\eps_0$. But it is still unlikely that the process $\ol{Y}_n$ will hit $-\de$ before hitting $\eps_0$ $j$ times. Therefore, it is likely that it will take a lot of time for $\ol{Y}_n$ to hit $-\de$. This completes the proof. 

\subsection{Formal outline of the proof of Lemma~\ref{lemma:master}}

Now, let us carry out the proof in detail. First, we outline the proof as a sequence of lemmata, which we later prove one by one. Assume without loss of generality that $\de \le \eps_0$ and $\eps_1 \le y_0 \le \eps_0$. Define for $n = 1, 2, \ldots$ and $a \in \BR$:
$$
\tau_a^{(n)} := \inf\{t \ge 0\mid Y_n(t) = a\}.
$$
Without loss of generality, we can assume that $y_n = y_0$. Indeed, if $\tau^{(n)}_{y_0} \ge \tau_n$, then $\min_{t \in [0, \tau_n]}Y_n(t) \ge y_0 > 0 > -\de$, and there is nothing to prove. On the other hand, if $\tau^{(n)}_{y_0} < \tau_n$, then it suffices to prove the statement for the sequence of processes $Y_n\bigl(\cdot + \tau^{(n)}_{y_0}\bigr)$ instead of the original sequence of processes $Y_n$. But the new sequence of processes, all of which start from $y_0$, also has the same properties (a) - (g) from the statement of Lemma~\ref{lemma:master}.

\begin{lemma} There exists a sequence $(p(n))_{n \ge 1}$ of positive numbers depending only on 
$$
C_1, C_2, C_3, C_4, C_5, \eps_0, \eps_1, (\ol{\ga}_n)_{n \ge 1},
$$ 
such that $p(n) \to 0$ as $n \to \infty$, and 
\label{lemma:aux-important}
$$
\MP\left(\tau^{(n)}_{-\de}\wedge\tau_n \le \tau^{(n)}_{\eps_0}\right) \le p(n).
$$
\end{lemma}

Assuming we already proved this lemma, let us complete the proof of Lemma~\ref{lemma:master}. Let 
$$
\tau_{\eps_1}^{(n, j)} := \inf\{t \ge \tau_{\eps_0}^{(n, j)}\mid Y_n(t) = \eps_1\},\ 
j,\, n = 1, 2, \ldots;\  \tau_{\eps_1}^{(n, 0)} := 0;
$$
$$
\tau_{\eps_0}^{(n, j)} := \inf\{t \ge \tau_{\eps_1}^{(n, j-1)}\mid Y_n(t) = \eps_0\},
 \ \ j = 1, 2, \ldots;\ \ n = 1, 2, \ldots
$$
At time $\tau_{\eps_1}^{(n, j)}$, the process $Y_n$ returns for the first time to $\eps_1$ after it has $j$ times returned to $\eps_0$. At time $\tau_{\eps_0}^{(n, j)}$, the process $Y_n$ returns for the first time to $\eps_0$ after it has $j-1$ times returned to $\eps_1$. In particular, $\tau_{\eps_0}^{(n, 1)} = \tau_{\eps_0}^{(n)}$ is the first hitting time of $\eps_0$. 

\begin{lemma}
\label{lemma:1991}
There exist $c > 0$ and $p \in (0, 1)$ such that for all $n = 1, 2, \ldots$ we have:
$$
\MP(\tau_{\eps_0}^{(n)} \le c\wedge\tau_n\wedge\tau_{-\eps_0}^{(n)}) \le p.
$$
\end{lemma}

This proves that there exists a copy of a Bernoulli random variable $\zeta_1$ such that for some $p_1 \le p$, 
\begin{equation}
\label{eq:distribution-of-eta}
\MP(\zeta_1 = c) = 1 - p_1,\ \MP(\zeta_1 = 0) = p_1,\ \ \mbox{and}\ \ \zeta_1\wedge\tau^{(n)}_{-\eps_0}\wedge\tau_n \le \tau^{(n)}_{\eps_0}\wedge\tau^{(n)}_{-\eps_0}\wedge\tau_n\ \ \mbox{a.s.}
\end{equation}

Let us prove a similar statement for the moment of the $j$th return to the level $\eps_0$: There exists a sequence $\zeta_1, \zeta_2, \ldots$ of Bernoulli random variables with distribution 
$$
\MP(\zeta_j = c) = 1 - p_j,\ \ \MP(\zeta_j = 0) = p_j,\ \ p_j \le p,\ \ j = 1, 2, \ldots,
$$
such that for each $n = 1, 2, \ldots$ and $j = 1, 2, \ldots$ we have:
\begin{equation}
\label{eq:comparison-of-times}
\left(\tau^{(n, j)}_{\eps_0} - \tau^{(n, j-1)}_{\eps_1}\right)\wedge\tau^{(n)}_{-\eps_0}\wedge\tau_n \ge \zeta_n\wedge\tau^{(n)}_{-\eps_0}\wedge\tau_n.
\end{equation}
Indeed, the moment $\left(\tau^{(n, j)}_{\eps_0} - \tau^{(n, j-1)}_{\eps_1}\right)\wedge\tau^{(n)}_{-\eps_0}\wedge\tau_n$ plays the same role for the process $Y_n(\cdot + \tau^{(n, j-1)}_{\eps_1})$ as the role of $\tau^{(n)}_{\eps_0}\wedge\tau^{(n)}_{-\eps_0}\wedge\tau_n$ for the original process $Y_n$. The sequence $(Y_n(\cdot + \tau^{(n, j-1)}_{\eps_1}))_{n \ge 1}$ has the same properties (a) - (g) as the process $Y_n$ and starts from $y_n = y_0 = \eps_1$. These random variables $\zeta_1, \zeta_2, \ldots$ can be taken to be independent, because the driving Brownian motions for the processes $Y_n(\cdot + \tau^{(n, j-1)}_{\eps_1})$ for different $j = 1, 2, \ldots$ are all independent (this follows from the Markov property of a Brownian motion). From~\eqref{eq:comparison-of-times} and $\tau^{(n, j-1)}_{\eps_1} \ge \tau^{(n, j-1)}_{\eps_0}$ it follows that 
\begin{equation}
\label{eq:new-comparison-of-times}
\left(\tau^{(n, j)}_{\eps_0} - \tau^{(n, j-1)}_{\eps_0}\right)\wedge\tau^{(n)}_{-\eps_0}\wedge\tau_n \ge \zeta_n\wedge\tau^{(n)}_{-\eps_0}\wedge\tau_n.
\end{equation}
Summing up~\eqref{eq:new-comparison-of-times} for $\tilde{j} = 1, \ldots, j$, and using the (easily checked) fact that $(a_1 + \ldots + a_j)\wedge c \le a_1\wedge c + \ldots + a_j\wedge c$ for $a_1, \ldots, a_j, c \ge 0$, we get:
\begin{equation}
\label{eq:comp-sequence-sum}
\tau^{(n, j)}_{\eps_0}\wedge\tau^{(n)}_{-\eps_0}\wedge\tau_n \ge \left(\zeta_1 + \ldots + \zeta_j\right)\wedge\tau^{(n)}_{-\eps_0}\wedge\tau_n.
\end{equation}
For $n, j = 1, 2, \ldots$ let $q_j := \MP\left(\zeta_1 + \ldots + \zeta_j \le T\right)$, and $p_j(n) := \MP\left(\tau^{(n)}_{-\de} \le \tau^{(n, j)}_{\eps_0}\wedge\tau_n\right)$. 

\begin{lemma} For every $n, j = 1, 2, \ldots$ we have: $\MP\bigl(\min\limits_{t \in [0, \tau_n]}Y_n(t) \le -\eta\bigr) \le q_j + p_j(n)$. 
\label{lemma:comp-of-probabilities}
\end{lemma}

\begin{lemma}
For every $n, j = 1, 2, \ldots$ we have: $p_j(n) \le jp(n)$. 
\label{lemma:pjp1}
\end{lemma}

To finish the proof of Lemma~\ref{lemma:master}, combine Lemmata~\ref{lemma:comp-of-probabilities},~\ref{lemma:pjp1}, and get: for every $j, n = 1, 2, \ldots$
$$
\MP\Bigl(\min\limits_{t \in [0, \tau_n]}Y_n(t) \le -\eta\Bigr) \le q_j + jp(n). 
$$
We have: $\zeta_i \ge 0$ a.s. and $\zeta_1, \zeta_2, \ldots$ are independent random variables with $\ME \zeta_i = 1 - p_i > 1 - p > 0$, $i = 1, 2, \ldots$ It is an easy exercise to show that $q_j \to 0$ as $j \to \infty$. Now, fix a small number $\xi > 0$. Take $j$ large enough so that $q_j < \xi/2$, and an $n_0$ so that for $n \ge n_0$ we have: $p(n) < \xi/(2(j+1))$. For $n \ge n_0$, 
$$
\MP\Bigl(\min\limits_{t \in [0, \tau_n]}Y_n(t) \le -\eta\Bigr) \le \frac{\xi}2 + \frac{\xi}2 = \xi.
$$
This completes the proof of Lemma~\ref{lemma:master}. 

\subsection{Proof of Lemma~\ref{lemma:aux-important}} 

We apply time-change to turn the martingale term into a Brownian motion. Then we compare the drift coefficient (which is not exactly a function of $Y_n(t)$) to a certain function of $Y_n(t)$. This allows us to compare the time-changed process $Y_n$ with the  solution to a certain SDE. Finally, we apply the scale function to the new process and complete the proof. For $t \le \tau_n\wedge\tau^{(n)}_{-\eps_0}\wedge\tau^{(n)}_{\eps_0}$, we have: $|Y_n(t)| \le \eps_0$, and $C_1 \le |\rho_n(t)| \le C_2$. Consider the process
$$
M_n(t) := \int_0^t\rho_n(s)\,\md B_n(s).
$$
The stopped process $M_n(\cdot\wedge\tau_n\wedge\tau^{(n)}_{-\de}\wedge\tau^{(n)}_{\eps_0})$ is a square integrable martingale, with nondecreasing quadratic variation
$$
v_n(t) := \langle M_n(\cdot\wedge\tau_n\wedge\tau^{(n)}_{-\eps_0}\wedge\tau^{(n)}_{\eps_0})\rangle_t = \int_0^{t\wedge\tau_n\wedge\tau^{(n)}_{-\eps_0}\wedge\tau^{(n)}_{\eps_0}}\rho^2_n(s)\,\md s.
$$
Let $\eta_n(s) := \inf\{t \ge 0\mid v_n(t) \ge s\}$ be the inverse function. It is well-defined for $s \in [0, s_n]$, where $s_n := \langle M_n\rangle_{\tau_n\wedge\tau^{(n)}_{-\eps_0}\wedge\tau^{(n)}_{\eps_0}}$, since the function $t \mapsto \langle M_n\rangle_t$ is strictly increasing on $[0, \tau_n\wedge\tau^{(n)}_{-\eps_0}\wedge\tau^{(n)}_{\eps_0}]$. Then by \cite[Lemma 2]{MyOwn1}, for some Brownian motion $W_n = (W_n(t), t \ge 0)$, and for the time-changed process $\ol{Y}_n(s) \equiv Y_n(\eta_n(s))$, $s \le s_n$, we get:
\begin{equation}
\label{eq:437}
\md \ol{Y}_n(s) = \frac{\ga_n(\eta_n(s))}{\rho^2_n(\eta_n(s))}\md s + \md W_n(s).
\end{equation}
Now, for $s \le s_n$ we have: $\ga_n(\eta_n(s)) \ge \ol{\ga}_n(Y_n(\eta_n(s))) = \ol{\ga}_n(\ol{Y}_n(s))$. Let us estimate the drift in~\eqref{eq:437}:
$$
\frac{\ga_n(\eta_n(s))}{\rho^2_n(\eta_n(s))} \ge 
\begin{cases}
C_2^{-2}\ol{\ga}_n(\ol{Y}_n(s)),\ \ \mbox{if}\ \ \ol{\ga}_n(\ol{Y}_n(s)) \ge 0;\\
-C_3C_1^{-2},\ \ \mbox{if}\ \ \ol{\ga}_n(\ol{Y}_n(s)) < 0.
\end{cases}
$$
We can combine these two cases as follows. Denote
\begin{equation}
\label{eq:psi-n}
\psi_n(x) := C_2^{-2}\ol{\ga}_n(x) - C_1^{-2}C_3. 
\end{equation}
Then we get:
$$
\frac{\ga_n(\eta_n(s))}{\rho^2(\eta_n(s))} \ge \psi_n(\ol{Y}_n(s)),\ \ s \le s_n.
$$
Let $\tilde{Y}_n = (\tilde{Y}_n(s), s \ge 0)$ be the solution to the following equation:
$$
\md \tilde{Y}_n(s) = \psi_n(\tilde{Y}_n(s))\md s + \md W_n(s),\ \ \tilde{Y}_n(0) = y_0.
$$
For $n = 1, 2, \ldots$ and $a \in \BR$, define the following stopping times:
$$
\ol{\tau}^{(n)}_a := \inf\{t \ge 0\mid \ol{Y}_n(t) = a\},\ \ \tilde{\tau}^{(n)}_a := \inf\{t \ge 0\mid \tilde{Y}_n(t) = a\}.
$$
It is straightforward to show that for $a \in [-\eps_0, \eps_0]$, we have: 
\begin{equation}
\label{eq:stopping-times-relation}
\ol{\tau}^{(n)}_a\wedge v_n(\tau_n) = v_n(\tau^{(n)}_a\wedge\tau_n). 
\end{equation}
By comparison theorem, see for example \cite[Section 6.1]{IWBook}, we get: 
\begin{equation}
\label{eq:comparison-Y-tilde-Y}
\ol{Y}_n(s) \ge \tilde{Y}_n(s)\ \ \mbox{for}\ \ s \le s'_n := s_n\wedge\tilde{\tau}^{(n)}_{\eps_0}\wedge\tilde{\tau}^{(n)}_{-\eps_0}.
\end{equation}
Since we can compare the processes $\tilde{Y}_n$ and $\ol{Y}_n$, we can also compare their hitting times. 

\begin{lemma} We have: $\left\{\tau^{(n)}_{-\de} < \tau^{(n)}_{\eps_0}\wedge\tau_n\right\} \subseteq \left\{\tilde{\tau}^{(n)}_{-\de} < \tilde{\tau}^{(n)}_{\eps_0}\right\}$. 
\end{lemma}

\begin{proof}
Assume $\tau^{(n)}_{-\de} < \tau^{(n)}_{\eps_0}\wedge\tau_n$. The function $v_n$ is strictly increasing on $[0, s_n]$. From~\eqref{eq:stopping-times-relation}, we have:
\begin{equation}
\label{eq:comp-100}
\ol{\tau}^{(n)}_{-\de} < \ol{\tau}^{(n)}_{\eps_0}\wedge v_n(\tau_n). 
\end{equation}
Let us show that $\tilde{\tau}^{(n)}_{-\de} < \tilde{\tau}^{(n)}_{\eps_0}$. Assume the converse. Denote 
$$
\ol{s}_n := \langle M_n\rangle_{\tau_n}\wedge\ol{\tau}^{(n)}_{\eps_0}\wedge\ol{\tau}^{(n)}_{-\de}\wedge\tilde{\tau}^{(n)}_{\eps_0}\wedge\tilde{\tau}^{(n)}_{-\de} \le s'_n. 
$$
From~\eqref{eq:comp-100} and our assmption, we have: $\ol{s}_n = \ol{\tau}^{(n)}_{-\de}\wedge\tilde{\tau}^{(n)}_{\eps_0}$. Now, consider the following two cases:

\smallskip

{\it Case 1:} $\ol{\tau}^{(n)}_{-\de} \le \tilde{\tau}^{(n)}_{\eps_0}$. Then 
$\ol{s}_n = \ol{\tau}^{(n)}_{-\de}$. Now, $\ol{Y}_n\left(\ol{s}_n\right) = -\de$. 
Therefore, by~\eqref{eq:comparison-Y-tilde-Y} we have: $\tilde{Y}_n\left(\ol{s}_n\right) \le -\de$. Hence $\tilde{\tau}^{(n)}_{-\de} \le \ol{s}_n \le \tilde{\tau}^{(n)}_{\eps_0}$. But $\ol{\tau}^{(n)}_{-\de} = \ol{\tau}^{(n)}_{\eps_0}$ cannot happen, because hitting times of two different levels by the same process cannot coincide. Therefore, we arrive at a contradiction. 

\smallskip

{\it Case 2:} $\ol{\tau}^{(n)}_{-\de} \ge \tilde{\tau}^{(n)}_{\eps_0}$. Then 
$\ol{s}_n = \tilde{\tau}^{(n)}_{\eps_0}$. Now, $\tilde{Y}_n\left(\ol{s}_n\right) = \eps_0$. Therefore, by~\eqref{eq:comparison-Y-tilde-Y} we have: $\ol{Y}_n\left(\ol{s}_n\right) = \eps_0$. Hence $\ol{\tau}^{(n)}_{\eps_0} \le \ol{s}_n \le \ol{\tau}^{(n)}_{-\de}$, which contradicts~\eqref{eq:comp-100}. 
\end{proof}

\begin{lemma} As $n \to \infty$, we have: $p(n) := \MP\left(\tilde{\tau}^{(n)}_{-\de} < \tilde{\tau}^{(n)}_{\eps_0}\right) \to 0$. 
\label{lemma:p-n-to-0}
\end{lemma}

\begin{rmk} Note that this probability depends only on the function $\psi_n$, which, in turn, depends only on the function $\ol{\ga}_n$ and constants $C_1, C_2, C_3$.
\end{rmk}

\begin{proof} The process $\tilde{Y}_n$ is a diffusion with scale function 
$$
S_n(x) := \int_{a_n}^xe^{-2\int_{b_n}^y\psi_n(z)\md z}\md y.
$$
One can choose $a_n$ and $b_n$ arbitrarily; we shall choose $a_n = y_0$, and $b_n$ from Lemma~\ref{lemma:b-n}. The sequence of functions $(\ol{\ga}_n)_{n \ge 1}$ has the property (e) from the statement of Lemma~\ref{lemma:master}. It is easy to show that the sequence $(\psi_n)_{n \ge 1}$ also has this property. Also, from the property (f) of $(\ol{\ga}_n)_{n \ge 1}$ follows a similar property of $(\psi_n)_{n \ge 1}$, but with a different constant $C_6 := C_2^{-2}C_4 - C_1^{-2}C_3$: $\sup_{y \in [b_n, \eps_0]}\psi_n(y)  \le C_6$. Take the sequence $(b_n)_{n \ge 1}$ from Lemma~\ref{lemma:b-n} (stated and proved below), applied to $g_n := \psi_n$. This choice of $b_n$ guarantees that $C_6$ is well-defined and independent of $n$. 

\begin{lemma} 
\label{lemma:scale-function}
We have:
\begin{equation}
\label{eq:x>0}
\varlimsup\limits_{n \to \infty}S_n(x) < \infty\ \ \mbox{for}\ \ x \in (0, \eps_0];
\end{equation}
\begin{equation}
\label{eq:x<0}
\lim\limits_{n \to \infty}S_n(x) = -\infty\ \ \mbox{for}\ \ x \in [-\eps_0, 0).
\end{equation}
\end{lemma}

\begin{proof} It suffices to show~\eqref{eq:x>0} for $x \ge y_0$, because each
$S_n$ is strictly increasing.  For $x \in [y_0, \eps_0]$, by Fatou's lemma, we get:
$$
\varlimsup\limits_{n \to \infty}S_n(x) \le \int_{y_0}^x\varlimsup\limits_{n \to \infty}e^{2C_6(y-b_n)}\md y \le \int_{y_0}^xe^{2C_6y}\md y < \infty,
$$
which proves~\eqref{eq:x>0}. To prove~\eqref{eq:x<0}, it suffices to show that 
\begin{equation}
\label{eq:limit-rel}
\int_y^{b_n}\psi_n(z)\md z \to \infty,\ \ y \in [-\eps_0, 0).
\end{equation}
Indeed, then we can apply Fatou's lemma and finish the proof. But~\eqref{eq:limit-rel} follows from the property (e) of the sequence of functions $(\psi_n)_{n \ge 1}$ and from elementary calculations:
$$
\int_y^{b_n}\psi_n(z)\md z = \int_{-|y|}^{|y|}\psi_n(z)\md z - \int_{b_n}^{|y|}\psi_n(z)\md z \ge \int_{-|y|}^{|y|}\psi_n(z)\md z - |y|\cdot C_7 \to \infty.
$$
\end{proof}

Let us complete the proof of Lemma~\ref{lemma:aux-important}. The probabiltiy that the diffusion process $\tilde{Y}_n$, starting from $y_0$, hits $-\de$ before $\eps_0$, is equal to 
$$
\MP\left(\tilde{\tau}^{(n)}_{-\de} < \tilde{\tau}^{(n)}_{\eps_0}\right) = 
\frac{S_n(\eps) - S_n(y_0)}{S_n(\eps) - S_n(-\de)} \to 0,\ \ n \to \infty,
$$
because by Lemma~\ref{lemma:scale-function}, 
$$
\varlimsup\limits_{n \to \infty}S_n(\eps_0) < \infty,\ \ \varlimsup\limits_{n \to \infty}S_n(y_0) < \infty, \ \ \mbox{and}\ \ \lim\limits_{n \to \infty}S_n(-\de) = -\infty. 
$$
This completes the proof of Lemma~\ref{lemma:p-n-to-0}, together with the proof of Lemma~\ref{lemma:aux-important}. 
\end{proof}

\begin{lemma} Assume $(g_n)_{n \ge 1}$ is a sequence of functions $g_n : [-\eps_0, \eps_0] \to \BR$ such that 
\begin{equation}
\label{eq:limsup}
\varlimsup\limits_{n \to \infty}\sup\limits_{y \in [b, \eps_0]}g_n(y) \le C_0\ \ \mbox{for every}\ \ b \in (0, \eps_0).
\end{equation}
Then there exists a sequence $(b_n)_{n \ge 1}$ of positive real numbers such that $b_n \to 0$, and 
$$
\varlimsup\limits_{n \to \infty}\sup\limits_{y \in [b_n, \eps_0]}g_n(y) \le C_0.
$$
\label{lemma:b-n}
\end{lemma}

\begin{proof} The proof is very similar to that of  \cite[Lemma 3.11]{MyOwn7}. For $b \in (0, \eps_0)$ and $n = 1, 2, \ldots$ let $M(b, n) := \sup_{y \in [b, \eps_0]}g_n(y)$. 
For every $k = 1, 2, \ldots$ there exists an $n_k$ such that for $n \ge n_k$, we have:
$$
\sup\limits_{y \in [k^{-1}, \eps_0]}g_n(y) \le C_0 + k^{-1}.
$$
Without loss of generality, we can take the sequence $(n_k)_{k \ge 1}$ to be strictly increasing: $n_{k} < n_{k+1}$. Define $(b_n)_{n \ge 1}$ as follows: $b_{n} = k^{-1}$ for $n_{k} \le  n < n_{k+1}$. For $n < n_1$, just let $b_n := 1$. Then we get: for $n_{k} \le  n < n_{k+1}$, $M(b_n, n) = M(k^{-1}, n) \le C_0 + b_n$. Since $b_n \to 0$, this completes the proof. 
\end{proof}

\subsection{Proof of Lemma~\ref{lemma:1991}}

We prove that for a certain standard Brownian motion $W_n$, 
\begin{equation}
\label{eq:subset-event}
\left\{\tau_{\eps_0}^{(n)} \le c\wedge\tau_n\wedge\tau_{-\de}^{(n)}\right\}
\subseteq \left\{\oa(W_n, [0, C_2^2T], C_2^2c) \ge \eps_0 - \eps_1 - cC_5\right\}.
\end{equation}
Assume that the following event happened:
$$
\tau_{\eps_0}^{(n)} \le c\wedge\tau_n\wedge\tau^{(n)}_{-\eps_0}.
$$
By continuity of $Y_n$, there exist $t_1$ and $t_2$ (dependent on $\oa \in \Oa$ and on $n$) such that $0 \le t_1 < t_2 \le c\wedge\tau_n\wedge\tau^{(n)}_{-\eps_0}$, and, in addition, $Y_n(t) \in [\eps_1, \eps_0]$ for $t \in [t_1, t_2]$, $Y_n(t_1) = \eps_1$, and $Y_n(t_2) = \eps_0$. Then we have:
$$
Y_n(t_2) - Y_n(t_1) = \int_{t_1}^{t_2}\ga_n(s)\md s + \int_{t_1}^{t_2}\rho_n(s)\md B_n(s).
$$
But $\ga_n(s) \le C_5$ for $s \in [t_1, t_2]$: this follows from the assumption (g) of Lemma~\ref{lemma:master}. Therefore, we get:
$$
\int_{t_1}^{t_2}\rho_n(s)\md B_n(s) \ge \eps_0 - \eps_1 - (t_2 - t_1)C_5 \ge \eps_0 - \eps_1 - cC_5. 
$$
We can make a time-change (see this carried in more detail earlier, in the proof of Lemma~\ref{lemma:aux-important}): for some standard Brownian motion $W_n = (W_n(s), s \ge 0)$, 
$$
M_n(t) := \int_0^t\rho_n(s)\md B_n(s) = W_n\left(\langle M_n\rangle_t\right),\ \ 
t \le \tau_n\wedge\tau^{(n)}_{-\eps_0}\wedge\tau^{(n)}_{\eps_0}.
$$
Also, we have the following estimate:
$$
\langle M_n\rangle_{t_2} - \langle M_n\rangle_{t_1} = \int_{t_1}^{t_2}\rho^2(s)\md s
\le C_2^2(t_2 - t_1) \le C_2^2c.
$$
Therefore, we have proved~\eqref{eq:subset-event}. Let $c \downarrow 0$. Then a.s. $\oa(W_n, [0, C_2^2T], C_2^2c) \to 0$. Meanwhile, $\eps_0 - \eps_1 - cC_5 \to \eps_0 - \eps_1 > 0$. Therefore,
$$
\MP\left(\oa(W_n, [0, C_2^2T], C_2^2c) \ge \eps_0 - \eps_1 - cC_5\right) \to 0.
$$
It suffices to note that the probability in the left-hand side does not depend on $n$.

\subsection{Proof of Lemma~\ref{lemma:comp-of-probabilities}}
Assume that the event $\{\min_{t \in [0, T\wedge\tau_n]}Y_n(t) \le -\de\}$ happened. Then we have: $\tau^{(n)}_{-\de} \le \tau_n$. Consider two cases:

\smallskip

{\it Case 1:} $\tau^{(n)}_{-\de}\wedge\tau_n \le \tau^{(n, j)}_{\eps_0}\wedge\tau_n$.
Therefore, $\tau^{(n)}_{-\de}\wedge\tau_n \le \tau^{(n, j)}_{\eps_0}$. But $\MP\left(\tau^{(n)}_{-\de}\wedge\tau_n \le \tau^{(n, j)}_{\eps_0}\right) \le p_j(n)$. 

% This means that the process $Y_n$ hits the level $-\de$ before it hits the level $\eps_0$ for the $j$th time. Therefore, for the process $Y_n\bigl(\cdot + \tau^{(n, j-1)}_{\eps_1}\bigr)$, which has the same properties (a) - (g) as the process $Y_n$ and starts from $y_n = y_0 = \eps_0$, we have: $\tau^{(n)}_{-\de}\wedge\tau_n \le \tau^{(n)}_{\eps_0}\wedge\tau_n$. (Now we apply the $\tau$ notation to the process $Y_n(\cdot + \tau^{(n, j)}_{\eps_1})$ instead of $Y_n$. This is a slight abuse of notation.) By Lemma~\ref{lemma:aux-important}, this event has probability less than or equal to $p(n)$.  

\smallskip

{\it Case 2:} $\tau^{(n)}_{-\de}\wedge\tau_n > \tau^{(n, j)}_{\eps_0}\wedge\tau_n$. It follows from $\tau^{(n)}_{-\de} \le \tau_n$ that $\tau^{(n, j)}_{\eps_0}\wedge\tau_n < \tau_n$. We can rewrite this as $\tau^{(n, j)}_{\eps_0} < \tau_n$. Therefore,
$\tau^{(n, j)}_{\eps_0} < \tau^{(n)}_{-\de}\wedge\tau_n$. Comparing with~\eqref{eq:comp-sequence-sum}, we get:
$$
\left(\zeta_1 + \ldots + \zeta_j\right)\wedge\tau^{(n)}_{-\eps_0}\wedge\tau_n \le T\wedge\tau^{(n)}_{-\de}\wedge\tau_n \le T \wedge\tau^{(n)}_{-\eps_0}\wedge\tau_n.
$$
Therefore, $\zeta_1 + \ldots + \zeta_j \le T$. It suffices to note that $\MP(\zeta_1 + \ldots + \zeta_j \le T) = q_j$.  

%\smallskip
%
%{\it Case 2.3.} $\tau^{(n)}_{-\de} < \tau^{(n, j)}_{\eps_0}\wedge\tau_n$. Loosely speaking, this event means that until the process $Y_n$ returns to the level $\eps_0$ for $j$ times, it hits the level $-\de$. This event has probability less than or equal to $p_j(n)$. 

\subsection{Proof of Lemma~\ref{lemma:pjp1}}
Apply induction by $j$. Base is trivial: $p_1(n) \le p(n)$. Induction step:
\begin{align*}
\MP&\left(\tau^{(n)}_{-\de} \le \tau^{(n, j)}_{\eps_0}\wedge\tau_n\right)  = 
\MP\left(\tau^{(n, j-1)}_{\eps_0} < \tau^{(n)}_{-\de} \le \tau^{(n, j)}_{\eps_0}\wedge\tau_n\right) + \MP\left(\tau^{(n)}_{-\de} \le \tau^{(n, j-1)}_{\eps_0}\wedge\tau_n\right) \\ & = \MP\left(\tau^{(n, j-1)}_{\eps_0} < \tau^{(n)}_{-\de} \le \tau^{(n, j)}_{\eps_0}\wedge\tau_n\right) + p_{j-1}(n).
\end{align*}
Assume that the following inequality holds:
\begin{equation}
\label{eq:comp-again}
\tau^{(n, j-1)}_{\eps_0} < \tau^{(n)}_{-\de}.
\end{equation}
Then $\tau^{(n, j-1)}_{\eps_1} < \tau^{(n)}_{-\de}$, because the process $Y_n$, starting from $\eps_0$ at the moment $\tau^{(n, j-1)}_{\eps_0}$, has to hit $\eps_1 \in (0, \eps_0)$ before it reaches the level $-\de$. The process $\tilde{Y}_n(\cdot) := Y_n(\cdot + \tau^{(n, j)}_{\eps_1})$ also satisfies the conditions of Lemma~\ref{lemma:master}. The moment $\tilde{\tau}^{(n)}_{\eps_0}$ for this new process is actually the moment $\tau^{(n, j)}_{\eps_0} - \tau^{(n, j)}_{\eps_1}$ for the original process $Y_n(\cdot)$. Meanwhile, the moment $\tilde{\tau}^{(n)}_{-\de}$ for this new process $\tilde{Y}_n$ satisfies $\tilde{\tau}^{(n)}_{-\de} = \tau^{(n)}_{-\de} - \tau^{(n, j-1)}_{\eps_1}$, if only the process $Y_n$ did not hit $-\de$ before $\tau^{(n, j)}_{\eps_1}$. But we assume this has not happened, because of~\eqref{eq:comp-again}. Therefore, we can apply results of Lemma~\ref{lemma:aux-important} and get:
$$
\MP\left(\tau^{(n, j-1)}_{\eps_0} < \tau^{(n)}_{-\de} \le \tau^{(n, j)}_{\eps_0}\wedge\tau_n\right) \le \MP\left(\tilde{\tau}^{(n)}_{-\de} \le \tilde{\tau}^{(n)}_{\eps_0}\wedge\tau_n\right) \le p(n). 
$$
Thus, $p_j(n) \le p_{j-1}(n) + p(n)$. The rest is trivial.

\section{Proof of Theorem~\ref{thm:3}} Let us show first that for every compact subset $\CK \subseteq \CU_0$, we have: 
\begin{equation}
\label{eq:tech}
\sup_{x \in \CK}\norm{r(y(x))} < \infty.
\end{equation}
Indeed, fix an $x_0 \in \CK$. Then for $x \in \CK$, we have:
$$
\norm{y(x)} \le \norm{y(x) - x} + \norm{x} \le \norm{y(x_0) - x}  + \norm{x} \le \norm{y(x_0) - x_0} + \norm{x_0} + 2\norm{x}.
$$
Taking $\sup$ over $x \in \CK$, we prove that $\sup_{x \in \CK}\norm{y(x)} < \infty$. Because $y$ is continuous on $\CU_0$, this completes the proof of~\eqref{eq:tech}.

\smallskip

Apply Theorems~\ref{thm:1} and~\ref{thm:2}. It suffices to show:

\smallskip

(a) for a compact subset $\CK \subseteq \CU$, the sequence $(\cm_{\CK, n})_{n \ge 1}$ has a spike at zero;

\smallskip

(b) for a compact subset $\CK \subseteq D$, $f_n \to 0$ uniformly on $\CK$;

\smallskip

(c) the sequence $(f_n)_{n \ge 1}$ emulates the reflection field $r$.

\smallskip

{\it Proof of (a):} For $z \in \pa D$, by Cauchy-Schwartz inequality we have: 
\begin{equation}
\label{eq:r-inequality}
\norm{r(z)} = \norm{r(z)}\norm{\fn(z)} \ge r(z)\cdot\fn(z) = 1.
\end{equation}
If $x \in \CK$ and $|\phi(x)| < \de_{\CK}$, then by Lemma~\ref{lemma:aux} $x \in \CU_0$, and $f_n(x) = g_n(\phi(x))r(y(x))$. Applying~\eqref{eq:r-inequality}, we get:
\begin{equation}
\label{eq:619}
\norm{f_n(x)} = |g_n(\phi(x))|\norm{r(y(x))} \ge g_n(\phi(x)).
\end{equation}
Now, let $s \in (-\de_{\CK}, \de_{\CK})$. Taking the $\inf$ over $x \in \CK$ such that $\phi(x) = s$ in~\eqref{eq:619}, and using~\eqref{eq:m-k-n}, we have: $\cm_{\CK, n}(s) \ge g_n(s)$. Since the sequence $(g_n)_{n \ge 1}$ has a spike at zero, $(\cm_{\CK, n})_{n \ge 1}$ does too.

\smallskip

{\it Proof of (b):} Since a continuous function $\phi$ is strictly positive on $\CK$, we have:
$$
0 < C_- := \inf_{z \in \CK}\phi(z) \le \sup_{z \in \CK}\phi(z) =: C_+ < \infty.
$$
Since $g_n \to 0$ uniformly on $[C_-, C_+]$, we have: $\sup_{x \in \CK}\norm{g_n(\phi(x))} \to 0$. Combining this observation with~\eqref{eq:tech}, we complete the proof. 

\smallskip

{\it Proof of (c):} It follows from the definition of $f_n$ from~\eqref{eq:f-n-example} that for $z \in \CU$, we have: 
$$
\frac{f_n(z)}{\norm{f_n(z)}} = \frac{r(y(z))}{\norm{r(y(z))}}.
$$
It suffices to note that $\CK(\de_{\CK}) \subseteq \CU$ by Lemma~\ref{lemma:aux}. The rest is trivial.

\section{Appendix}

In this section, we state some auxillary lemmata on tightness and weak convergence. They were used in the proofs earlier in the article. Some of the lemmata have straightforward proofs, which are omitted. Other lemmata have a bit more complicated proofs. 

\begin{lemma}
\label{lemma:conv-of-ito}
Take a sequence $(\xi_n)_{n \ge 0}$ of adapted processes $\xi_n = (\xi_n(t), 0 \le t \le T)$, which are bounded by a universal constant: $|\xi_n(t)| \le C$ for all $t \ge 0$ and $n = 0, 1, 2, \ldots$ Take a sequence of standard Brownian motions $(W_n)_{n \ge 0}$. Assume $\xi_n \to \xi_0$ a.s. for almost all $t \in [0, T]$ as $n \to \infty$, and $W_n \to W_0$ a.s. uniformly on $[0, T]$ as $n \to \infty$. Define 
$$
M_n(t) := \int_0^t\xi_n(s)\,\md W_n(s),\ \ n = 1, 2, \ldots,\ \ t \in [0, T].
$$
Then we have the following convergence as $n \to \infty$:
$$
\ME\max\limits_{0 \le t \le T}\left(M_n(t) - M_0(t)\right)^2 \to 0.
$$
As a corollary, for every random moment $\tau$ in $[0, T]$ (not necessarily a stopping time)  we have the following convergence in $L^2$: $M_n(\tau) \to M_0(\tau)$. 
\end{lemma}

\begin{proof} We can represent
\begin{align*}
\ME&(M_n(T) - M_0(T))^2 = \ME\left(\int_0^T\xi_n\,\md W_n(t) - \int_0^T\xi_0(t)\,\md W_0(t)\right)^2 \\ & \le 2\ME\left(\int_0^T\left(\xi_n-\xi_0\right)\,\md W_n(t)\right)^2 + 2\ME\left(\int_0^T\xi_0(t)\,\md\left(W_n(t) - W_0(t)\right)\right)^2 \\ & \le 
2\int_0^T\ME\left(\xi_n(t) - \xi_0(t)\right)^2\,\md t + 2\ME\int_0^T\xi_0^2(t)\,\md\langle W_n - W_0\rangle_t.
\end{align*}
The first term tends to zero because of Lebesgue dominated convergence theorem (applied twice, to the time integral and the expectation). The second term: since $W_n - W_0$ is a continuous square-integrable martingale, by Burkholder-Davis-Gundy inequality, \cite[Chapter 3, Theorem 3.28]{KSBook}, 
\begin{equation}
\label{eq:BDG}
\ME\langle W_n - W_0\rangle_T \le C_0\,\ME(W_n(T) - W_0(T))^2.
\end{equation}
Here, $C_0$ is some universal constant. Also, $W_n(T) \to W_0(T)$ a.s. To prove that 
\begin{equation}
\label{eq:0065}
\ME (W_n(T) - W_0(T))^2 \to 0,
\end{equation}
we need to show that the family $((W_n(T) - W_0(T))^2)_{n \ge 1}$ is uniformly integrable. But this is true, because $(a + b)^4 \le 8(a^4 + b^4)$ for all $a, b \in \BR$, and therefore
$$
\ME(W_n(T) - W_0(T))^4 \le 8\left(\ME W_n^4(T) + \ME W_0^4(T)\right) = 8\left(3T^2 + 3T^2\right) = 48T^2 < \infty.
$$
This proves~\eqref{eq:0065}. From~\eqref{eq:BDG} and~\eqref{eq:0065}, we get: $\ME\langle W_n - W_0\rangle_T \to 0$. Thus,
$$
\ME\int_0^T\xi_0^2(t)\,\md\langle W_n - W_0\rangle_t \le C^2\ME\langle W_n - W_0\rangle_T \to 0.
$$
This proves that $\lim_{n \to \infty}\ME(M_n(T) - M_0(T))^2 = 0$. But $M_n - M_0$ is a square-integrable martingale. Applying Burkholder-Davis-Gundy inequality again,  \cite[Chapter 3, Theorem 3.28]{KSBook}, we have:
$$
\ME\max\limits_{0 \le t \le T}(M_n(t) - M_0(t))^2 \le K\,\ME(M_n(T) - M_0(T))^2 \to 0.
$$
\end{proof}

\begin{lemma} Suppose $h : [0, T] \to \BR^d$ is continuous, and $l_n, l$ are continuous and nondecreasing functions $[0, T] \to \BR$ with $l_n(0) = l(0) = 0$. Assume also that $l_n \to l$ uniformly on $[0, T]$. Then 
$$
\int_0^Th(s)\md l_n(s) \to \int_0^Th(s)\md l(s),\ \ \mbox{as}\ \ n \to \infty.
$$
\label{lemma:weak-like-conv}
\end{lemma}

\begin{proof} Without loss of generality, assume $l_n(T) = l(T) = 1$. Otherwise, we can always divide $l_n(\cdot)$ by $l_n(T)$, and $l(\cdot)$ by $l(T)$, and use $l_n(T) \to l(T)$. 
Then one can consider $l_n$ and $l$ as a cumulative distribution functions of random variables $\xi_n, \xi \in [0, T]$. We have: $\xi_n \Ra \xi$ as $n \to \infty$. Since $h$ is continuous and therefore bounded on $[0, T]$, we have:
$$
\ME h(\xi_n) = \int_0^Th(s)\md l_n(s) \to \ME h(\xi) = \int_0^Th(s)\md l(s).
$$
\end{proof}

The following lemmata follow from straightforward applications of the Arzela-Ascoli theorem.

\begin{lemma}
\label{lemma:three-processes-tight}
Take three sequences of $C[0, T]$-valued random variables 
$$
(X_n)_{n \ge 1},\ (X'_n)_{n \ge 1},\ (X''_n)_{n \ge 1}.
$$
Assume that initial conditions are zero:
$$
X_n(0) = X'_n(0) = X''_n(0) = 0,\ \ n = 1, 2,\ldots
$$
Assume also that for all $0 \le s \le t \le T$, $n = 1, 2, \ldots$
$$
X''_n(t) - X''_n(s) \le X_n(t) - X_n(s) \le X'_n(t) - X'_n(s).
$$
If the sequences $(X'_n)_{n \ge 1}$ and $(X''_n)_{n \ge 1}$ are tight, then $(X_n)_{n \ge 1}$ is also tight.
\end{lemma}

\begin{lemma}
\label{lemma:bounded-tight}
Take a sequence of $\BR^d$-valued a.s. integrable processes $(\be_n(t), 0 \le t \le T)$. 
If the following sequence of processes is tight in $C[0, T]$:
$$
t \mapsto \int_0^t\norm{\be_n(s)}\md s,\ \ n = 1, 2, \ldots
$$
then the following sequence of processes is tight in $C([0, T], \BR^d)$:
$$
t \mapsto \int_0^t\be_n(s)\md s,\ \ n = 1, 2, \ldots
$$
\end{lemma}

 \begin{lemma}
\label{lemma:cont-mapping-tight}
Take a continuous function $\psi : \BR \to \BR$. If the sequence of processes $(X_n)_{n \ge 1}$ is tight in $C[0, T]$, then the sequence $(\psi(X_n(\cdot)))_{n \ge 1}$ is also tight in $C[0, T]$. 
\end{lemma}

\begin{lemma}
\label{lemma:stopping-tight}
Take a tight sequence $(X_n)_{n \ge 1}$ of processes in $C([0, T], \BR^d)$. Let $(\tau_n)_{n \ge 1}$ be a sequence of random moments with values in $[0, T]$. Then $(X_n(\cdot\wedge\tau_n))_{n \ge 1}$ is also a tight sequence.
\end{lemma}

\section*{Acknowledgements}

The author would like to thank \textsc{Cameron Bruggeman}, \textsc{Krzysztof Burdzy}, \textsc{Jean-Pierre Fouque},  \textsc{Ioannis Karatzas}, \textsc{Soumik Pal}, \textsc{Kavita Ramanan}, \textsc{Mykhaylo Shkolnikov}, and \textsc{Ruth Williams} for help and useful discussion. The author would like to thank the reviewer and editor for help and useful remarks. This research was partially supported by  NSF grants DMS 1007563, DMS 1308340, DMS 1405210, and  DMS 1409434.

\medskip\noindent

\bibliographystyle{alpha}

\bibliography{oblique}

\end{document}